\documentclass[a4paper]{amsart}
\usepackage{amssymb}
\usepackage[all]{xy}
\usepackage{hyperref,aliascnt}
\usepackage{enumerate}

\numberwithin{equation}{section}

\newtheorem{lma}{Lemma}[section]

\newaliascnt{thmCt}{lma}
\newtheorem{thm}[thmCt]{Theorem}
\aliascntresetthe{thmCt}

\newaliascnt{corCt}{lma}
\newtheorem{cor}[corCt]{Corollary}
\aliascntresetthe{corCt}

\newaliascnt{prpCt}{lma}
\newtheorem{prp}[prpCt]{Proposition}
\aliascntresetthe{prpCt}

\newcounter{theoremintro}

\newtheorem{thmIntro}[theoremintro]{Theorem}
\newtheorem{corIntro}[theoremintro]{Corollary}

\theoremstyle{definition}

\newaliascnt{pgrCt}{lma}
\newtheorem{pgr}[pgrCt]{}
\aliascntresetthe{pgrCt}

\newaliascnt{dfnCt}{lma}
\newtheorem{dfn}[dfnCt]{Definition}
\aliascntresetthe{dfnCt}

\newaliascnt{rmkCt}{lma}
\newtheorem{rmk}[rmkCt]{Remark}
\aliascntresetthe{rmkCt}

\newaliascnt{qstCt}{lma}
\newtheorem{qst}[qstCt]{Question}
\aliascntresetthe{qstCt}

\newaliascnt{exaCt}{lma}
\newtheorem{exa}[exaCt]{Example}
\aliascntresetthe{exaCt}

\DeclareMathOperator{\Sub}{Sub}
\DeclareMathOperator{\Cu}{Cu}
\newcommand{\axiomD}[1]{(D#1)}
\newcommand{\axiomO}[1]{(O#1)}
\newcommand{\NN}{\mathbb{N}}
\newcommand{\KK}{\mathcal{K}}
\newcommand{\ca}{$C^*$-al\-ge\-bra}
\newcommand{\SubSep}{\Sub_{\mathrm{sep}}}
\newcommand{\SubCtbl}{\Sub_{\mathrm{ctbl}}}
\newcommand{\freeVar}{\_\,}
\newcommand{\txtSup}{\mathrm{sup}}
\newcommand{\txtSeq}{\mathrm{seq}}
\newcommand{\andSep}{\,\,\,\text{ and }\,\,\,}

\newcommand{\CuSgp}{$\mathrm{Cu}$-sem\-i\-group}
\newcommand{\CuMor}{$\mathrm{Cu}$-mor\-phism}

\title{Covering dimension of Cuntz semigroups II}

\author{Hannes Thiel}
\address{H.\ Thiel,
Institute of Geometry, TU Dresden, Zellescher Weg 12-14, 01069 Dresden, Germany.}
\email{hannes.thiel@posteo.de}
\urladdr{www.hannesthiel.org}

\author{Eduard Vilalta}
\address{E.\ Vilalta,
Departament de Matem\`{a}tiques, 
Universitat Aut\`{o}noma de Barcelona, 08193 Bellaterra, Barcelona, Spain}
\email{evilalta@mat.uab.cat}

\thanks{The first named author was partially supported by the ERC Consolidator Grant No. 681207.
The second named author was partially supported by MINECO (grant No.\ PRE2018-083419 and No.\ MTM2017-83487-P), and by the Comissionat per Universitats i Recerca de la Generalitat de Catalunya (grant No.\ 2017SGR01725).
}
\subjclass[2010]%
{Primary
46L05, 
46L85; 
Secondary
54F45, 
55M10. 
}
\keywords{$C^*$-algebras, Cuntz semigroups, covering dimension}
\date{\today}

\begin{document}

\begin{abstract}
We show that the dimension of the Cuntz semigroup of a \ca{} is determined by the dimensions of the Cuntz semigroups of its separable sub-\ca{s}.
This allows us to remove separability assumptions from previous results on the dimension of Cuntz semigroups.

To obtain these results, we introduce a notion of approximation for abstract Cuntz semigroups that is compatible with the approximation of a \ca{} by sub-\ca{s}.
We show that many properties for Cuntz semigroups are preserved by approximation and satisfy a L{\"o}wenheim-Skolem condition.
\end{abstract}

\maketitle

\section{Introduction}

The Cuntz semigroup of a \ca{} $A$ encodes the comparison theory of positive elements in $A$ and its stabilization in a partially ordered, abelian monoid $\Cu(A)$.
This invariant was introduced by Cuntz \cite{Cun78DimFct} in his pioneering work on simple \ca{s}, and it continues to play an important role to this day.
For example, it was used by Toms to distinguish his groundbreaking examples of nonisomorphic simple, nuclear \ca{s} with the same $K$-theoretic data \cite{Tom08ClassificationNuclear}, to classify algebras and morphisms in \cite{Rob12LimitsNCCW}, and it was a key feature in some recent breakthroughs in the structure theory of \ca{s} \cite{Thi20RksOps, AntPerRobThi18arX:CuntzSR1}.
 
In \cite{ThiVil21arX:DimCu}, we introduced a notion of covering dimension (see \autoref{dfn:dim}) for Cuntz semigroups and their abstract counterparts, the \CuSgp{s} as introduced in \cite{CowEllIva08CuInv} and extensively studied in
\cite{AntPerThi18TensorProdCu, AntPerThi20AbsBivariantCu}.
Among other results, we proved the expected permanence properties (recalled in \autoref{prp:Permanence}), studied the relation between the dimension of $\Cu(A)$ and the nuclear dimension of $A$, and computed the dimension of Cuntz semigroups of simple, $\mathcal{Z}$-stable \ca{s}.

The goal of this paper is to further develop the results from \cite{ThiVil21arX:DimCu} and provide additional tools to compute the dimension of Cuntz semigroups and \CuSgp{s}.

Our first main result is a new permanence property:
the dimension of Cuntz semigroups behaves well with respect to approximation by sub-\ca{s}.
Here, we say that a \ca{} $A$ is \emph{approximated} by a collection of sub-\ca{s} $A_\lambda\subseteq A$ if for every $a_1,\ldots,a_n\in A$ and $\varepsilon>0$ there exist $\lambda\in\Lambda$ and $b_1,\ldots,b_n\in A_\lambda$ such that $\|b_j-a_j\|<\varepsilon$ for $j=1,\ldots,n$.
(This is stronger than requiring that $\bigcup_\lambda A_\lambda$ is dense in $A$.
On the other hand, the subalgebras are not required to be nested.)
For example, a \ca{} is \emph{locally finite-dimensional} (sometimes called \emph{locally AF}) if and only if it is approximated by a family of finite-dimensional sub-\ca{s}.

\begin{thmIntro}[\ref{prp:DimApproxCAlg}]
\label{thmIntro:A}
Let $A$ be a \ca{} that is approximated by a family of sub-\ca{s} $A_\lambda\subseteq A$. 
Then $\dim(\Cu(A))\leq \sup_\lambda \dim(\Cu(A_\lambda))$.
\end{thmIntro}

To prove this result, we introduce a notion of approximation for \CuSgp{s} (see \autoref{dfn:ApproxCu}) and we show that if a \ca{} $A$ is approximated by a family of sub-\ca{s} $A_\lambda\subseteq A$, then $\Cu(A)$ is approximated by the corresponding Cuntz semigroups $\Cu(A_\lambda)$;
see \autoref{prp:CuApproxCAlg}.
For every fixed $n$, we prove that the property of having dimension at most~$n$ is preserved by approximations (see \autoref{prp:ApproxDim}), which then gives Theorem~\ref{thmIntro:A} above.

Our second main result shows that the dimension of Cuntz semigroups satisfies the L{\"o}wenheim-Skolem condition.

\begin{thmIntro}[\ref{prp:SepSubCaDim}]
\label{thmIntro:B}
Let $A$ be a \ca. 
Then, for every separable sub-\ca{} $B_0\subseteq A$ there exists a separable sub-\ca{} $B\subseteq A$ such that $B_0\subseteq B$ and $\dim(\Cu(B))\leq\dim(\Cu(A))$.
\end{thmIntro}

By combining Theorems~\ref{thmIntro:A} and~\ref{thmIntro:B}, we obtain the following characterization of the dimension of the Cuntz semigroup of a \ca{} in terms of its separable sub-\ca{s}:

\begin{corIntro}[\ref{prp:CharDimCa}]
Let $A$ be a \ca{}, and let $n\in\NN$.
Then $\dim(\Cu(A))\leq n$ if and only if every finite (or countable) subset of $A$ is contained in a separable sub-\ca{} $B\subseteq A$ satisfying $\dim(\Cu(B))\leq n$.
\end{corIntro}

The permanence properties from \cite{ThiVil21arX:DimCu} together with Theorems~\ref{thmIntro:A} and~\ref{thmIntro:B} show that associating to each \ca{} the dimension of its Cuntz semigroup is a well-behaved invariant that satisfies all but one property of a noncommutative dimension theory;
see \autoref{pgr:DimThy}.
The failing property is the compatibility with minimal unitizations;
see \autoref{exa:IncrDimUnitization}.
It remains open if associating to each \ca{} the dimension of the Cuntz semigroup of its minimal unitization is well-behaved;
see \autoref{qst:DimCuIsDimThy}.

To prove Theorem~\ref{thmIntro:B}, we show that every \ca{} $A$ admits a large collection of separable sub-\ca{s} $B\subseteq A$ such that the induced map $\Cu(B)\to\Cu(A)$ is an order-embedding (see \autoref{prp:SubCaWithSubCu}) and we prove the L{\"o}wenheim-Skolem condition for the dimension of \CuSgp{s}:
Given a \CuSgp{} $S$ and a countably based sub-\CuSgp{} $T_0\subseteq S$, there exists a countably based sub-\CuSgp{} $T\subseteq S$ such that $T_0\subseteq T$ and $\dim(T)\leq\dim(S)$;
see \autoref{prp:subCuDim}.

In \autoref{sec:SubCu} we investigate when a submonoid $T$ of a \CuSgp{} $S$ is a sub-\CuSgp{}.
In analogy to topological derived sets and Cantor-Bendixson derivatives, we introduce the associated \CuSgp{s} $T'$ and $\delta (T)$ (see Definitions~\ref{dfn:Derived} and~\ref{dfn:Delta}). 
In particular, we show that $T\subseteq S$ is a sub-\CuSgp{} if and only if $T=T'$; 
see \autoref{prp:CharSubCu3}. 
Further, using the sub-\CuSgp{} $\delta(T)$, we prove that the sub-\CuSgp{s} of a \CuSgp{} form a complete lattice when ordered by inclusion; 
see \autoref{prp:LatticeSubCu}.

Our results also provide a characterization of the dimension of a \CuSgp{} through its countably based sub-\CuSgp{s}:

\begin{thmIntro}[\ref{prp:CharDimCtbl}]
\label{thmIntro:CharDim}
Let $S$ be a \CuSgp, and let $n\in\NN$.
Then $\dim(S)\leq n$ if and only if every finite (or countable) subset of $S$ is contained in a countably based sub-\CuSgp{} $T\subseteq S$ satisfying $\dim(T)\leq n$.
\end{thmIntro}

As an application of Theorem~\ref{thmIntro:CharDim}, we generalize some results from \cite{ThiVil21arX:DimCu} by removing the assumption of countable basedness; 
see Propositions~\ref{prp:Dim0ImplInterpol} and~\ref{prp:DimSoftPart}.

\section{Preliminaries}

In the next paragraphs, we briefly recall the definition of (abstract) Cuntz semigroups.
We refer to \cite{AraPerTom11Cu} and \cite{AntPerThi18TensorProdCu} for details.

\begin{pgr}
\emph{The Cuntz semigroup.}
Given a \ca{} $A$, we use $A_+$ to denote the set of its positive elements.
For $a,b\in A_+$, one says that $a$ is \emph{Cuntz subequivalent} to $b$, in symbols $a\precsim b$, if  $a=\lim_n r_nbr_n^*$ for some sequence $(r_n)_n$ in $A$. 
One also writes $a\sim b$, and says that $a$ is \emph{Cuntz equivalent} to $b$, if $a\precsim b$ and $b\precsim a$.

The \emph{Cuntz semigroup} of $A$ is the set of equivalence classes $\Cu(A):=(A\otimes \KK)_+/\sim$, where $A\otimes\KK$ denotes the stabilization of $A$. One endows $\Cu(A)$ with the partial order induced by $\precsim$.
Further, addition of orthogonal elements in $(A\otimes\KK)_+$ induces an abelian monoid structure on $\Cu(A)$.
This turns $\Cu(A)$ into a positively ordered monoid, that is, every element $x\in\Cu(A)$ satisfies $0\leq x$, and if $x,y,z\in\Cu(A)$ satisfy $x\leq y$, then $x+z\leq y+z$.

Given $a\in A\otimes\KK$, we denote its class in $\Cu (A)$ by $[a]$.
\end{pgr}

\begin{pgr}
\emph{Abstract Cuntz semigroups.}
In \cite{CowEllIva08CuInv} it was shown that, beyond being a positively ordered monoid, the Cuntz semigroup of a \ca{} always satisfies four additional properties.
To formulate them, we need to recall the \emph{way-below relation}:
An element $x$ in a partially ordered set is said to be \emph{way-below} (or \emph{compactly contained} in) $y$, denoted by $x\ll y$, if for every increasing sequence $(z_n)_n$ that has a supremum $z$ satisfying $y\leq z$ there exists $n\in\NN$ such that $x\leq z_n$.

The properties introduced in \cite{CowEllIva08CuInv}, and that the Cuntz semigroup of a \ca{} always satisfies, are:
\begin{enumerate}
\item[\axiomO{1}] 
Every increasing sequence has a supremum.
\item[\axiomO{2}] 
Every element is the supremum of a $\ll$-increasing sequence.
\item[\axiomO{3}] 
Given $x'\ll x$ and $y'\ll y$, we have $x'+y'\ll x+y$.
\item[\axiomO{4}] 
Given increasing sequences $(x_n)_n$ and $(y_n)_n$, we have $\sup_n x_n +\sup_n y_n= \sup_n (x_n +y_n)$.
\end{enumerate}

Moreover, it was also proved in \cite{CowEllIva08CuInv} that every *-homomorphism $\varphi\colon A\to B$ between two \ca{s} $A$ and $B$ induces an order-preserving monoid morphism $\Cu(A)\to\Cu(B)$ that also preserves suprema of increasing sequences and the way-below relation.

It follows that the Cuntz semigroup defines a functor from the category of \ca{s} and *-homomorphisms to the category $\Cu$ of \CuSgp{s} and \CuMor{s}, which are defined as follows:
A \emph{\CuSgp} (also called \emph{abstract Cuntz semigroup}) is a positively ordered monoid satisfying \axiomO{1}-\axiomO{4}.
A \emph{\CuMor} between \CuSgp{s} $S$ and $T$ is an order-preserving monoid morphism $S\to T$ that preserves suprema of increasing sequences and the way-below relation.
\end{pgr}

\begin{pgr}
\emph{Additional properties.}
In addition to \axiomO{1}-\axiomO{4}, the following properties are known to be satisfied by the Cuntz semigroup of every \ca{} (see \cite[Proposition~4.6]{AntPerThi18TensorProdCu}, \cite{Rob13Cone} and \cite[Proposition~2.2]{AntPerRobThi19arX:Edwards} respectively):
\begin{enumerate}
\item[\axiomO{5}] 
Given $x+y\leq z$, $x'\ll x$ and $y'\ll y$, there exists $c$ such that $x'+c\leq z\leq x+c$ and $y'\ll c$.
\item[\axiomO{6}]
Given $x'\ll x\leq y+z$ there exist $v\leq x,y$ and $w\leq x,z$ such that $x'\leq v+w$.
\item[\axiomO{7}]
Given $x_1'\ll x_1\leq w$ and $x_2'\ll x_2\leq w$ there exists $x$ such that $x_1',x_2'\ll x\leq w, x_1+x_2$.
\end{enumerate}

It is common to use \axiomO{5} when $y=0$, that is, for $x'\ll x\leq z$.
In this case, \axiomO{5} implies that there exists an element $c$ such that $x'+c\leq z\leq x+c$.

A Cuntz semigroup is said to be \emph{weakly cancellative} if, whenever $x+z\ll y+z$, we have $x\ll y$. 
It was shown in \cite[Theorem~4.3]{RorWin10ZRevisited} that stable rank one \ca{s} have weakly cancellative Cuntz semigroups.
\end{pgr}

The following result contains a characterization of \axiomO{5} that will be used in \autoref{prp:ApproxO5O6O7} to show that \axiomO{5} is preserved by approximation of \CuSgp{s}, and in \autoref{prp:LS-O5O6O7} to show that it satisfies the L{\"o}wenheim-Skolem condition.
Analogous characterizations of \axiomO{6} and \axiomO{7} are shown in Propositions~\ref{prp:CharO6} and~\ref{prp:CharO7} below.

Recall that a subset $B$ of a \CuSgp{} $S$ is said to be a \emph{basis} if for every $x',x\in S$ satisfying $x'\ll x$ there exists $y\in B$ such that $x'\ll y\ll x$.
A \CuSgp{} is said to be \emph{countably based} if it contains a countable basis.

Cuntz semigroups of separable \ca{s} are countably based (see, for example, \cite[Lemma~1.3]{AntPerSan11PullbacksCu}).

\begin{prp}
\label{prp:CharO5}
Let $S$ be a \CuSgp{}. 
Then $S$ satisfies \axiomO{5} if and only if there exists a basis $B\subseteq S$ with (equivalently, every basis $B\subseteq S$ has) the following property: 
for all $x',x,y',y,z',z\in B$ satisfying
\[
x+y\ll z', \quad
x'\ll x, \quad 
y'\ll y, \quad
z'\ll z,
\]
there exists $c\in B$ such that
\[
x'+c\ll z, \quad
z'\ll x+c, \andSep
y'\ll c.
\]
\end{prp}
\begin{proof}
First, assume that $S$ satisfies \axiomO{5}, and let $B\subseteq S$ be a basis.
To verify that $B$ has the stated property, let $x',x,y',y,z',z\in B$ satisfy $x+y\ll z'$, $x'\ll x$, $y'\ll y$ and $z'\ll z$.
Choose $z''\in S$ satisfying $z'\ll z''\ll z$.
Applying \axiomO{5} for $x+y\leq z''$, $x'\ll x$ and $y'\ll y$, we obtain $a\in S$ such that
\[
x'+a\leq z''\leq x+a, \andSep y'\ll a.
\]
Using that $z'\ll z''$ and $y'\ll a$, choose $a'\in S$ such that
\[
z'\ll x+a', \andSep
y'\ll a' \ll a.
\]
Since $B$ is a basis, we obtain $c\in B$ with $a'\ll c\ll a$.
Then $c$ has the desired properties.

Next, assume that $B\subseteq S$ is a basis with the stated property.
The proof is similar to that of \cite[Theorem~4.4(1)]{AntPerThi18TensorProdCu}.
To verify that $S$ satisfies \axiomO{5}, let $x',x,y',y,z$ be elements in $S$ such that
\[
x+y\leq z,\quad 
x'\ll x, \andSep 
y'\ll y.
\]

Since $B$ is a basis, we can choose a $\ll$-decreasing sequence $(x_n)_n$ in $B$ such that
\[
x'\ll \ldots \ll x_{n+1} \ll x_n \ll \ldots \ll x_1 \ll x_0 \ll x.
\]
Further, we can choose $c_0',c_0\in B$ with
\[
y'\ll c_0'\ll c_0 \ll y.
\]
Using that $x_0+c_0\ll z$, we can take a $\ll$-increasing sequence $(z_n)_n$ in $B$ with supremum $z$, and such that $x_0+c_0\ll z_0$.

We have
\[
x_0+c_0\ll z_0, \quad
x_1\ll x_0, \quad
c_0'\ll c_0, \andSep
z_0\ll z_1.
\]
By assumption, we obtain $c_1\in B$ such that
\[
x_1+c_1 \ll z_1, \quad
z_0 \ll x_0+c_1, \andSep
c_0'\ll c_1.
\]
Choose $c_1'\in B$ such that $z_0\ll x_0+c_1'$ and $c_0'\ll c_1'\ll c_1$.

Then
\[
x_1+c_1\ll z_1, \quad
x_2\ll x_1, \quad
c_1'\ll c_1, \andSep
z_1\ll z_2.
\]
By assumption, we obtain $c_2\in B$ such that
\[
x_2+c_2 \ll z_2, \quad
z_1 \ll x_1+c_2, \andSep
c_1'\ll c_2.
\]
Choose $c_2'\in B$ such that $z_1\ll x_1+c_2'$ and $c_1'\ll c_2'\ll c_2$.

Proceeding in this manner inductively, we obtain a $\ll$-increasing sequence $(c_n')_n$ such that
\[
x'+c_n'\leq x_n+c_n\ll z_n \leq z, \andSep
z_n\ll x_n+c_{n+1}'\leq x+c_{n+1}'
\]
for each $n$.
Therefore, the supremum $c:=\sup_n c_n'$ satisfies $x'+c\leq z\leq x+c$, as desired.
\end{proof}

\begin{prp}
\label{prp:CharO6}
Let $S$ be a \CuSgp{}. 
Then $S$ satisfies \axiomO{6} if and only if there exists a basis $B\subseteq S$ with (equivalently, every basis $B\subseteq S$ has) the following property: 
for all $x',x,y',y,z',z\in B$ satisfying
\[
x\ll y'+z', \quad
x'\ll x, \quad 
y'\ll y, \quad
z'\ll z,
\]
there exist $v,w\in B$ such that
\[
x'\ll v+w, \quad
v\ll x,y, \andSep
w\ll x,z.
\]
\end{prp}
\begin{proof}
Assuming that $S$ satisfies \axiomO{6}, one can use the same methods as in the proof of \autoref{prp:CharO5} to see that every basis $B\subseteq S$ satisfies the desired condition.

Next, assume that $B\subseteq S$ is a basis with the stated property.
To verify that $S$ satisfies \axiomO{6}, let $x',x,y,z\in S$ satisfy
\[
x'\ll x \leq y+z.
\]
Using that $B$ is a basis, choose $a',a\in B$ such that $x'\ll a'\ll a \ll x$. Thus, one has  $a\ll y+z$, and we can choose $b',b,c',c\in B$ satisfying
\[
a\ll b'+c', \quad
b'\ll b\ll y, \andSep
c'\ll c\ll z.
\]
By assumption, we obtain $v,w\in B$ such that 
\[
x'\ll a'\ll v+w, \quad
v\ll a,b, \andSep
w\ll a,c.
\]
Since $a\ll x$, $b\ll y$ and $c\ll z$, the elements $v$ and $w$ have the desired properties.
\end{proof}

The next results are proved with the same methods as \autoref{prp:CharO6}.
We omit the proofs.

\begin{prp}
\label{prp:CharO7}
Let $S$ be a \CuSgp{}. 
Then $S$ satisfies \axiomO{7} if and only if there exists a basis $B\subseteq S$ with (equivalently, every basis $B\subseteq S$ has) the following property: 
for all $x_1',x_1,x_2',x_2,w',w\in B$ satisfying
\[
x_1'\ll x_1\ll w', \quad
x_2'\ll x_2\ll w', \andSep
w'\ll w,
\]
there exists $x\in B$ such that
\[
x_1',x_2'\ll x\ll w,x_1+x_2.
\]
\end{prp}

\begin{prp}
\label{prp:CharCanc}
A \CuSgp{} $S$ is weakly cancellative if and only if there exists a basis $B\subseteq S$ with (equivalently, every basis $B\subseteq S$ has) the following property: 
for all $x',x,y',y,z',z\in B$ satisfying $x'\ll x$, $y'\ll y$ and $z'\ll z$ with $x+z\ll y'+z'$, we have $x' \ll y$.
\end{prp}

We recall the definition of (covering) dimension for \CuSgp{s} from \cite[Defintion~3.1]{ThiVil21arX:DimCu}:

\begin{dfn}
\label{dfn:dim}
Let $S$ be a \CuSgp.
Given $n\in\NN$, we write $\dim (S)\leq n$ if, whenever $x'\ll x\ll y_1+\ldots +y_r$ in $S$, then there exist $z_{j,k}\in S$ for $j=1,\ldots ,r$ and $k=0,\ldots ,n$ such that: 
\begin{enumerate}[(i)]
\item 
$z_{j,k}\ll y_j$ for each $j$ and $k$;
\item 
$x'\ll \sum_{j,k} z_{j,k}$;
\item 
$\sum_{j=1}^r z_{j,k}\ll x$ for each $k=0,\ldots,n$.
\end{enumerate}

We set $\dim(S)=\infty$ if there exists no $n\in\NN$ with $\dim (S)\leq n$.
Otherwise, we let $\dim(S)$ be the smallest $n\in\NN$ such that $\dim(S)\leq n$.
We call $\dim(S)$ the \emph{(covering) dimension} of $S$.
\end{dfn}

The following result summarizes the permanence properties shown in \cite{ThiVil21arX:DimCu}.

\begin{prp}[{\cite[Propositions~3.5, 3.7, 3.9]{ThiVil21arX:DimCu}}]
\label{prp:Permanence}
Given a \CuSgp{} $S$ and an ideal $I\subseteq S$, we have:
\[
\dim(I)\leq\dim(S), \andSep
\dim(S/I)\leq\dim (S).
\]
Given \CuSgp{s} $S$ and $T$, we have:
\[
\dim(S\oplus T) = \max\{\dim(S),\dim(T)\}.
\]

Given an  inductive limit of \CuSgp{s} $S=\varinjlim_{\lambda\in\Lambda} S_{\lambda}$, we have
\[
\dim(S)\leq \liminf_{\lambda} \dim(S_{\lambda}).
\]
\end{prp}

\section{Approximation of \texorpdfstring{$\mathrm{Cu}$}{Cu}-semigroups and \texorpdfstring{\ca{s}}{C*-algebras}}
\label{sec:Approx}

In this section, we introduce a notion of \emph{approximation} for a \CuSgp{} $S$ by a family of \CuMor{s} $S_\lambda\to S$;
see \autoref{dfn:ApproxCu}.
The definition ensures that any `reasonable' property passes to the approximated \CuSgp{}, and we show this specifically for the property of having dimension at most $n$;
see \autoref{prp:ApproxDim}.

If $S$ is an inductive limit of a system of \CuSgp{s} $S_\lambda$, then the canonical maps $S_\lambda\to S$ approximate $S$;
see \autoref{prp:ApproxLimit}.
Another natural source of approximation comes from \ca{s}:
If a \ca{} $A$ is approximated by a family of sub-\ca{s} $A_\lambda\subseteq A$ (we recall the definition before \autoref{prp:CuApproxCAlg}), then $\Cu(A)$ is approximated by the \CuMor{s} $\Cu(A_\lambda)\to\Cu(A)$ induced by the inclusions $A_\lambda\to A$;
see \autoref{prp:CuApproxCAlg}.

Note that in \autoref{dfn:ApproxCu} below we do not insist that the \CuMor{s} $S_\lambda\to S$ are order-embeddings.
The reason is that this would exclude the abovementioned sources of approximations.
Indeed, the natural maps $S_\lambda\to S$ to an inductive limit are not necessarily order-embeddings.
Further, if $A_\lambda\subseteq A$ is a sub-\ca{}, then the induced \CuMor{} $\Cu(A_\lambda)\to\Cu(A)$ need not be an order-embedding (consider, for example, $\mathbb{C}\subseteq\mathcal{O}_2$).

\begin{dfn}
\label{dfn:ApproxCu}
Let $S$ be a \CuSgp{} and let $(S_\lambda ,\varphi_\lambda)_{\lambda\in\Lambda}$ be a family of \CuSgp{s} $S_\lambda$ and \CuMor{s} $\varphi_\lambda\colon S_\lambda\to S$. 
 
We say that the family $(S_\lambda ,\varphi_\lambda)_{\lambda\in\Lambda}$ \emph{approximates} $S$ if the following holds:
Given finite sets $J$ and $K$, 
given elements $x_j',x_j\in S$ for $j\in J$, 
and given functions $m_k,n_k\colon J\to\NN$ for $k\in K$, 
such that $x_j'\ll x_j$ for all $j\in J$ 
and such that
\[
\sum_{j\in J} m_k(j)x_j \ll \sum_{j\in J} n_k(j)x_j'
\]
for all $k\in K$,
there exist $\lambda\in\Lambda$ 
and $y_j\in S_\lambda$ for $j\in J$ 
such that $x_j'\ll \varphi_\lambda(y_j) \ll x_j$ for each $j\in J$,
and such that
\[
\sum_{j\in J} m_k(j)y_j \ll \sum_{j\in J} n_k(j)y_j
\]
for all $k\in K$.
\end{dfn}

\begin{rmk}
\label{rmk:ApproxCu}
In \autoref{dfn:ApproxCu}, we think of $J$ as the index set for a collection of variables, and for each $k\in K$ we think of the pair $(m_k,n_k)$ as the encoding of a `formula'.
We say that $S$ is approximated by the $S_\lambda$ if every finite collection of elements in $S$ that satisfy certain formulas can be approximated by a collection of elements in some $S_\lambda$ that satisfy the same formulas.

Assume that the \CuSgp{} $S$ is approximated by the family $(S_\lambda ,\varphi_\lambda)_{\lambda\in\Lambda}$.
\autoref{dfn:ApproxCu} ensures that every `reasonable' property of \CuSgp{s} passes from the approximating family to $S$.
In \autoref{prp:ApproxO5O6O7} we show this for weak cancellation, \axiomO{5}, \axiomO{6} and \axiomO{7}, and in \autoref{prp:ApproxDim} we prove it for the property `$\dim(\freeVar)\leq n$'.


We do not formalize the notion of `formula' or `reasonable property' for \CuSgp{s} since this would go into the direction of developing a model theory for \CuSgp{s}, which is an elaborate task that will be taken up elsewhere.
\end{rmk}

\begin{prp}
\label{prp:ApproxO5O6O7}
Let $S$ be a \CuSgp{} that is approximated by $(S_\lambda ,\varphi_\lambda)_{\lambda\in\Lambda}$.
If each $S_\lambda$ is weakly cancellative, then so is $S$.
Similarly, if each $S_\lambda$ satisfies \axiomO{5} (respectively, \axiomO{6} or \axiomO{7}), then so does $S$.
\end{prp}
\begin{proof}
First, assume that each $S_\lambda$ is weakly cancellative. To see that $S$ is also weakly cancellative, we will use \autoref{prp:CharCanc}. Thus, let  $x',x,y',y,z',z\in S$ satisfy $x'\ll x$, $y'\ll y$, $z'\ll z$ and $x+z\ll y'+z'$.

Since $S$ is approximated by $(S_\lambda ,\varphi_\lambda)_{\lambda\in\Lambda}$, there exist $\lambda\in\Lambda$ and elements $u,v,w\in S_\lambda$ such that 
\[
x'\ll \varphi_\lambda (u)\ll x, \quad
y'\ll \varphi_\lambda (v)\ll y, \quad
z'\ll \varphi_\lambda (w)\ll z, \andSep 
u+w\ll v+w.
\]
Since $S_\lambda$ is weakly cancellative, one gets $u\ll v$ and, consequently,
\[
 x'\ll \varphi_\lambda (u)\ll \varphi_\lambda (v)\ll y,
\]
as required.

Next, assume that each $S_\lambda$ satisfies \axiomO{5}.
We show that $S$ satisfies the property of \autoref{prp:CharO5}.
Let $x',x,y',y,z',z\in S$ satisfy
\[
x+y\ll z', \quad
x'\ll x, \quad 
y'\ll y, \quad
z'\ll z.
\]
We need to find $c\in S$ such that
\[
x'+c\ll z, \quad
z'\ll x+c, \andSep
y'\ll c.
\]
By assumption, there exist $\lambda\in\Lambda$ and $u,v,w\in S_\lambda$ such that
\[
x'\ll\varphi_\lambda(u)\ll x, \quad
y'\ll\varphi_\lambda(v)\ll y, \quad
z'\ll\varphi_\lambda(w)\ll z, \andSep
u+v\ll w.
\]
Since $\varphi_\lambda$ is a \CuMor{} and $u+v\ll w$, we can choose $u',v'\in S_\lambda$ such that
\[
x'\ll \varphi_\lambda(u'), \quad
u'\ll u, \quad
y'\ll \varphi_\lambda(v'), \andSep
v'\ll v.
\]
Using that $S_\lambda$ satisfies \axiomO{5}, we obtain $a\in S_\lambda$ such that
\[
u'+a\leq w\leq u+a,\andSep 
v'\ll a.
\]
Then $c:=\varphi_\lambda(a)$ has the desired properties.

The statements for \axiomO{6} and \axiomO{7} are proved with similar methods, using Propositions~\ref{prp:CharO6} and~\ref{prp:CharO7} respectively.
\end{proof}

\begin{prp}
\label{prp:ApproxDim}
Let $S$ be a \CuSgp{} that is approximated by a family $(S_\lambda ,\varphi_\lambda)_{\lambda\in\Lambda}$.
Then $\dim(S)\leq\sup_{\lambda\in\Lambda} \dim (S_\lambda)$.
\end{prp}
\begin{proof}
Set $n:=\sup_{\lambda\in\Lambda} \dim (S_\lambda)$, which we may assume to be finite.
To verify $\dim(S)\leq n$, let $x'\ll x\ll y_1+\ldots +y_r$ in $S$.
Choose $y_1',\ldots,y_r'\in S$ such that
\[
x'\ll x\ll y_1'+\ldots +y_r', \quad
y_1'\ll y_1, \quad\ldots, \andSep
y_r'\ll y_r.
\]

Using that $S$ is approximated by $(S_\lambda ,\varphi_\lambda)_{\lambda\in\Lambda}$, we obtain $\lambda\in\Lambda$ and elements $v,w_1,\ldots ,w_r\in S_\lambda$ such that
\[
x'\ll \varphi_\lambda(v) \ll x, \quad
y_1'\ll \varphi_\lambda(w_1) \ll y_1, \quad\ldots, \andSep
y_r'\ll \varphi_\lambda(w_r) \ll y_r,
\]
and such that
\[
v\ll w_1+\ldots+w_r.
\]

Since $\varphi_\lambda$ is a \CuMor{} and $x'\ll \varphi_\lambda(v)$, there exists $v'\in S_\lambda$ such that 
\[
x'\ll \varphi_\lambda (v'), \andSep
v'\ll v.
\]

We have  $v'\ll v\ll w_1+\cdots +w_r$ in $S_\lambda$.
Using that $\dim(S_\lambda)\leq n$, we obtain elements $z_{j,k}\in S_\lambda$ for $j=1,\ldots ,r$ and $k=0,\ldots ,n$ satisfying conditions (i)-(iii) in \autoref{dfn:dim}.
It is now easy to check that the elements $\varphi_\lambda (z_{j,k})\in S$ satisfy conditions (i)-(iii) in \autoref{dfn:dim} for $x'\ll x\ll y_1+\ldots +y_r$, as desired.
\end{proof}

\begin{prp}
\label{prp:ApproxLimit}
Let $S=\varinjlim_{\lambda\in\Lambda}S_\lambda$ be an inductive limit of \CuSgp{s}, and let $\varphi_\lambda\colon S_\lambda\to S$ be the \CuMor{s} into the limit.
Then the family $(S_\lambda ,\varphi_\lambda)_{\lambda\in\Lambda}$ approximates $S$.
\end{prp}
\begin{proof}
For $\lambda\leq\mu$ in $\Lambda$, let $\varphi_{\mu,\lambda}\colon S_\lambda\to S_\mu$ denote the connecting \CuMor{} of the inductive system.
We will use the following conditions, which were shown in \cite[Paragraph~3.8]{ThiVil21arX:DimCu} to characterize that $S$ is the inductive limit:
\begin{enumerate}
\item[(L0)]
We have $\varphi_{\mu}\circ\varphi_{\mu,\lambda}=\varphi_{\lambda}$ for all $\lambda\leq\mu$ in $\Lambda$;
\item[(L1)]
If $x_\lambda\in S_\lambda$ and $x_\mu\in S_\mu$ satisfy $\varphi_{\lambda}(x_\lambda)\ll\varphi_{\mu}(x_\mu)$, then there exists $\mu$ with $\lambda,\mu\leq\nu$ such that $\varphi_{\nu,\lambda}(x_\lambda)\ll\varphi_{\nu,\mu}(x_\mu)$;
\item[(L2)]
For all $x',x\in S$ satisfying $x'\ll x$ there exist $\lambda\in\Lambda$ and $x_\lambda\in S_\lambda$ such that  $x'\ll\varphi_\lambda(x_\lambda)\ll x$.
\end{enumerate}

Let $J$ and $K$ be finite sets, let $x_j',x_j\in S$ satisfy $x_j'\ll x_j$ for $j\in J$, 
and let $m_k,n_k\colon J\to\NN$ such that
\[
\sum_{j\in J} m_k(j)x_j 
\ll \sum_{j\in J} n_k(j)x_j'
\]
for all $k\in K$.

For each $j\in J$, applying (L2), we obtain $\lambda_j\in\Lambda$ and $z_j\in S_{\lambda_j}$ such that
\[
x_j' \ll \varphi_{\lambda_j}(z_j) \ll x_j.
\]
Choose $\lambda\in\Lambda$ with $\lambda_j\leq\lambda$ for all $j$, and set $\bar{z}_j:=\varphi_{\lambda,\lambda_j}(z_j)\in S_{\lambda}$ for each $j$.

Given $k\in K$, we have
\[
\varphi_{\lambda}\left( \sum_{j\in J} m_k(j)\bar{z}_j \right)
\ll \sum_{j\in J} m_k(j)x_j
\ll \sum_{j\in J} n_k(j)x_j'
\ll \varphi_{\lambda}\left( \sum_{j\in J} n_k(j)\bar{z}_j \right).
\]
Applying (L1), we obtain $\nu_k\in\Lambda$ with $\lambda\leq\nu_k$ such that
\begin{align}
\label{prp:AproxLimit:eq1}
\varphi_{\nu_k,\lambda}\left( \sum_{j\in J} m_k(j)\bar{z}_j \right)
\ll \varphi_{\nu_k,\lambda}\left( \sum_{j\in J} n_k(j)\bar{z}_j \right).
\end{align}
Choose $\nu\in\Lambda$ with $\nu_k\leq\nu$ for all $k$, and set $y_j:=\varphi_{\nu,\lambda}(\bar{z}_j )\in S_{\nu}$ for each $j$.

For each $j$, we have
\[
\varphi_\nu(y_j) = \varphi_{\lambda}(\bar{z}_j ) = \varphi_{\lambda_j}(z_j)
\]
and therefore $x_j'\ll\varphi_\nu(y_j)\ll x_j$.
Further, for $k\in K$, using \eqref{prp:AproxLimit:eq1}, we obtain
\begin{align*}
\sum_{j\in J} m_k(j)y_j 
&= \varphi_{\nu,\nu_k}\left( \varphi_{\nu_k,\lambda}\left( \sum_{j\in J} m_k(j)\bar{z}_j \right) \right) \\
&\ll \varphi_{\nu,\nu_k}\left( \varphi_{\nu_k,\lambda}\left( \sum_{j\in J} n_k(j)\bar{z}_j \right) \right)
= \sum_{j\in J} n_k(j)y_j,
\end{align*}
as desired.
\end{proof}

The next result recovers \cite[Theorem~4.5]{AntPerThi18TensorProdCu} and \cite[Proposition~3.9]{ThiVil21arX:DimCu}, and is in fact new for \axiomO{7}.

\begin{cor}
\label{prp:LimitProperties}
Let $S=\varinjlim_{\lambda\in\Lambda}S_\lambda$ be an inductive limit of \CuSgp{s}.
If each $S_\lambda$ is weakly cancellative, then so is $S$.
Similarly, if each $S_\lambda$ satisfies \axiomO{5} (respectively, \axiomO{6} or \axiomO{7}), then so does $S$.
Further, given $n\in\NN$ such that $\dim(S_\lambda)\leq n$ for each $\lambda$, then $\dim(S)\leq n$.
\end{cor}
\begin{proof}
This follows by combining \autoref{prp:ApproxLimit} with Propositions~\ref{prp:ApproxO5O6O7} and~\ref{prp:ApproxDim}.
\end{proof}

A \ca{} $A$ is said to be \emph{approximated} by a collection of sub-\ca{s} $A_\lambda\subseteq A$, for $\lambda\in\Lambda$, if for every finitely many elements $a_1,\ldots,a_n\in A$ and every $\varepsilon>0$ there exist $\lambda\in\Lambda$ and $b_1,\ldots,b_n\in A_\lambda$ such that $\|b_j-a_j\|<\varepsilon$ for $j=1,\ldots,n$.

\begin{prp}
\label{prp:CuApproxCAlg}
Let $A$ be a \ca{} that is approximated by a family of sub-\ca{s} $A_\lambda\subseteq A$, and let $i_\lambda\colon A_\lambda\to A$ be the inclusion maps for $\lambda\in\Lambda$.
Then, the system $(\Cu (A_\lambda),\Cu (i_\lambda))_{\lambda\in\Lambda}$ approximates $\Cu(A)$.
\end{prp}
\begin{proof}
We may assume that $A$ and $A_\lambda$ are stable for every $\lambda\in\Lambda$.
We begin with three claims. Since they are simple computations, we omit their proof (for Claim~1, one can approximate the function $(t-\varepsilon )_+$ by a polynomial).

\textbf{Claim 1.}
\emph{For any $\varepsilon,\delta >0$ and $a\in A_+$, there exists $\sigma>0$ such that, whenever $b\in A_+$ satisfies $\Vert a-b \Vert\leq \sigma$, we have $\Vert (a-\varepsilon)_+ - (b-\varepsilon)_+ \Vert\leq \delta$.}

\textbf{Claim 2.}
\emph{Let $\varepsilon>0$ and let $a,b,r\in A$ be such that $\Vert a-rbr^*\Vert <\varepsilon$. 
Then, there exists $\delta>0$ such that for every $c,d,s\in A$ with 
\[
\Vert c-a\Vert<\delta, \quad 
\Vert d-b\Vert<\delta, \andSep
\Vert s-r\Vert<\delta
\]
one has
\[
\Vert c-sds^*\Vert <2\varepsilon .
\]
}

\textbf{Claim 3.}
\emph{
Given $a\in A_+$ and $\varepsilon>0$, there exists $\delta>0$ such that for every $b\in A$ with $\|b-a\|<\delta$ we have $\|(b^*b)^{1/2}-a\|<\varepsilon$.
}

 

Now let $J$ and $K$ be finite sets, and take elements $x_j',x_j\in \Cu(A)$ for $j\in J$, 
and functions $m_k,n_k\colon J\to\NN$ for $k\in K$, 
such that $x_j'\ll x_j$ for all $j\in J$ 
and such that
\[
\sum_{j\in J} m_k(j)x_j \ll \sum_{j\in J} n_k(j)x_j'
\]
for all $k\in K$. 
We may assume that $(\sum_{j\in J} m_k(j))(\sum_{j\in J} n_k(j))\neq 0$ for every $k$.

For each $j\in J$, let $a_j\in A_+$ be such that $[a_j]=x_j$. 
Since~$J$ is finite, there exists $\varepsilon >0$ such that 
\[
x_j'\ll [(a_j - 2\varepsilon )_+]\ll [a_j]=x_j
\]
for every $j\in J$.

For each $k\in K$, we have
\[
\left[ \bigoplus_{j\in J} a_j^{\oplus m_k(j)} \right]
=\sum_{j\in J} m_k(j)x_j 
\ll \sum_{j\in J} n_k(j)x_j'
\leq \left[ \bigoplus_{j\in J} (a_j-2\varepsilon )_+^{\oplus n_k(j)} \right],
\]
which allows us to take $r_k\in A$ satisfying
\[
\left\Vert \bigoplus_{j\in J} a_j^{\oplus m_k(j)} - r_k\left( \bigoplus_{j\in J} (a_j-2\varepsilon )_+^{\oplus n_k(j)} \right) r_k^*\right\Vert
< \frac{\varepsilon}{2}.
\]

For each $k$, let $\delta_{k}$ be the bound given by Claim 2 for the previous inequality, and take $\delta >0$ such that 
\[
\delta <\min_{k\in K} \frac{\delta_{k}}{(\sum_{j\in J} m_k(j))(\sum_{j\in J} n_k(j))},\andSep 
\delta < \varepsilon .
\]
Then, using Claim 1, let $\sigma>0$ satisfy $\sigma\leq \delta$ and such that for every $j\in J$ and $b\in A_+$ with $\Vert a_j-b\Vert\leq \sigma$, we have $\Vert (a_j-2\varepsilon )_+ - (b-2\varepsilon )_+\Vert\leq \delta$.

Since the sub-\ca{s} $A_\lambda$ approximate $A$, and using Claim~3, there exist $\lambda\in\Lambda$ and elements $s_k\in A_\lambda$ and $b_j\in (A_\lambda )_+$ such that
\[
\Vert s_k-r_k\Vert \leq\sigma ,\andSep 
\Vert b_j-a_j\Vert \leq\sigma
\]
for each $k\in K$ and $j\in J$. 

By the choice of $\sigma$, note that we also have $\Vert (b_j -2\varepsilon )_+ - (a_j -2\varepsilon )_+\Vert \leq \delta $ for each $j\in J$.
Using that $\Vert (b_j - \varepsilon)_+ - a_j\Vert <2\varepsilon $ in the first step, and that $\Vert b_j-a_j\Vert <\varepsilon$ in the second step, we note that the element $[(b_j -\varepsilon )_+]\in\Cu (A)$ satisfies
\[
[(a_j-2\varepsilon )_+]\ll [(b_j -\varepsilon )_+]
\ll [a_j]
\]
for every $j\in J$.

Further, for each $k\in K$, we have
\[
\left\Vert \bigoplus_{j\in J} a_j^{\oplus m_k(j)} - \bigoplus_{j\in J} b_j^{\oplus m_k(j)} \right\Vert 
\leq \sum_{j\in J} m_k(j) \Vert a_j-b_j \Vert 
\leq \sum_{j\in J} m_k (j) \delta < \delta_k
\]
and, similarly, 
\[
\left\Vert \bigoplus_{j\in J} (a_j-2\varepsilon )_+^{\oplus n_k(j)} -
\bigoplus_{j\in J}(b_j-2\varepsilon )_+^{\oplus n_k(j)} \right\Vert 
\leq \sum_{j\in J} n_k(j)\delta < \delta_k.
\]
Thus, it follows from Claim 2 that, for every $k\in K$, we get
\[
\left\Vert \bigoplus_{j\in J} b_j^{\oplus m_k(j)} 
- c_k\left( \bigoplus_{j\in J}(b_j-2\varepsilon )_+^{\oplus n_k(j)}\right) c_k^*\right\Vert
< 2\frac{\varepsilon}{2}
=\varepsilon
\]
and, consequently,
\[
\sum_{j\in J} m_k (j) [(b_j-\varepsilon )_+]
\leq \sum_{j\in J} n_k (j) [(b_j-2\varepsilon )_+]
\]
in $\Cu(A_\lambda )$.

Recall that $i_\lambda \colon A_\lambda\to A$ denotes the inclusion map. 
Using that $[(a_j-2\varepsilon )_+]\ll [(b_j -\varepsilon )_+]\ll [a_j]$ in $\Cu (A)$ and $[(b_j -2\varepsilon )_+]\ll [(b_j -\varepsilon )_+]$ in $\Cu (A_\lambda)$, one notes that the elements $[(b_j -\varepsilon )_+]\in \Cu (A_\lambda)$ satisfy 
\[
x_j'
\ll [(a_j-2\varepsilon )_+]
\ll \Cu (i_\lambda )([(b_j -\varepsilon )_+])
\ll [a_j]
= x_j
\]
for every $j\in J$, and
\[
\sum_{j\in J} m_k (j) [(b_j-\varepsilon )_+]
\leq \sum_{j\in J} n_k (j) [(b_j-2\varepsilon )_+]
\ll \sum_{j\in J} n_k (j) [(b_j-\varepsilon )_+]
\]
for every $k\in K$, as desired.
\end{proof}

\begin{thm}
\label{prp:DimApproxCAlg}
Let $A$ be a \ca{} that is approximated by a family of sub-\ca{s} $A_\lambda\subseteq A$, for $\lambda\in\Lambda$.
Then $\dim(\Cu(A))\leq\sup_{\lambda\in\Lambda}\dim(\Cu(A_\lambda))$.
\end{thm}
\begin{proof}
By \autoref{prp:CuApproxCAlg}, we know that the system $(\Cu (A_\lambda),\Cu (i_\lambda))_{\lambda\in\Lambda}$ approximates $\Cu (A)$. 
Thus, the result follows from \autoref{prp:ApproxDim}.
\end{proof}

\section{The lattice of \texorpdfstring{sub-$\mathrm{Cu}$}{sub-Cu}-semigroups}
\label{sec:SubCu}

In this section, we provide characterizations for when a submonoid of a \CuSgp{} is a sub-\CuSgp, which will be used in \autoref{sec:reduction}.
In particular, given a \CuSgp{} $S$, we construct for every submonoid $T\subseteq S$ an associated sup-closed submonoid $\overline{T}^{\txtSup}\subseteq S$ (see \autoref{dfn:SupClosure}) and a `derived' submonoid $T'\subseteq S$ (see \autoref{dfn:Derived}).
We show that a submonoid $T\subseteq S$ is a sub-\CuSgp{} if and only if $T=T'$;
see \autoref{prp:CharSubCu3}.

We also describe, for every submonoid $T\subseteq S$, the largest sub-\CuSgp{} contained in $\overline{T}^{\txtSup}$.
This construction is used in \autoref{prp:LatticeSubCu} to prove that the collection of sub-\CuSgp{s} of a \CuSgp{} is a complete lattice.

\begin{dfn}
Given a \CuSgp{} $S$, we say that a submonoid $T\subseteq S$ is a \emph{sub-\CuSgp{}} if $T$ is a \CuSgp{} with respect to the order induced by $S$ and if the inclusion map $T\to S$ is a \CuMor.
\end{dfn}

The next results provide characterizations of sub-\CuSgp{s}.
We omit the straightforward proofs.

\begin{lma}
\label{prp:CharSubCu1}
Let $S$ be a \CuSgp. 
Then a submonoid $T\subseteq S$ is a sub-\CuSgp{} if and only if it is closed under passing to suprema of increasing sequences and for every $x'\in S$ and $x\in T$ with $x'\ll x$ there exists $y\in T$ such that $x'\ll y\ll x$.
\end{lma}

\begin{lma}
\label{prp:CharSubCu2}
Let $S,T$ be \CuSgp{s}, and let $\varphi\colon T\to S$ be a \CuMor.
Then the following are equivalent:
\begin{enumerate}
\item
$\varphi$ is an order-embedding, that is, $x,y\in T$ satisfy $x\leq y$ if (and only if) $\varphi(x)\leq\varphi(y)$;
\item
$x,y\in T$ satisfy $x\ll y$ if (and only if) $\varphi(x)\ll\varphi(y)$;
\item
$\varphi(T)\subseteq S$ is a sub-\CuSgp{} and $\varphi\colon T\to\varphi(T)$ is an isomorphism.
\end{enumerate}
\end{lma}

\begin{rmk}
Let us recall the notion of subobjects from category theory.
Let $\mathcal{C}$ be a category, and let $X$ be an object in $\mathcal{C}$.
Given monomorphisms $\alpha\colon Y\to X$ and $\beta\colon Z\to X$, one sets $\alpha\sim\beta$ if there exists an isomorphism $\gamma\colon Y\to Z$ such that $\beta\circ\gamma=\alpha$.
This defines an equivalence relation on the class of monomorphisms to~$X$, and a subobject of $X$ is defined as an equivalence class of this relation.
We refer to \cite[Section~4.1]{Bor94HandbookCat1} for details.

Let $S$ be a \CuSgp, and let $T\subseteq S$ be a sub-\CuSgp.
It is easy to verify that the inclusion map $T\to S$ is a monomorphism in the category $\mathrm{Cu}$, whence every sub-\CuSgp{} of $S$ naturally is a subobject.
The converse holds if and only if the following question has a positive answer.
\end{rmk}

\begin{qst}
Is every monomorphism in the category $\mathrm{Cu}$ an order-embedding?
\end{qst}

\begin{dfn}
\label{dfn:SupClosure}
Let $S$ be a \CuSgp, and let $T\subseteq S$ be a subset.
We set
\[
\overline{T}^{\txtSeq} 
:= \left\{ \sup_n x_n \in S : (x_n)_n \text{ is an increasing sequence in } T \right\}.
\]

We define $\overline{T}^{(\alpha)}$ for every ordinal $\alpha$ by setting $T^{(0)} := T$, $T^{(1)} := \overline{T}^{\txtSeq}$, and by using (transfinite) induction:
\begin{align*}
\overline{T}^{(\alpha+1)} &:= \overline{T^{(\alpha)}}^{\txtSeq}, \\
\overline{T}^{(\lambda)} &:= \bigcup_{\alpha<\lambda}\overline{T}^{(\alpha)},\quad\text{ if $\lambda$ is a limit ordinal}.
\end{align*}

We define the \emph{sup-closure} of $T$ as $\overline{T}^{\txtSup}  := \bigcup_{\alpha\geq 1} \overline{T}^{(\alpha)}$.
We say that $T$ is \emph{sup-closed} if $T=\overline{T}^{\txtSup}$.
\end{dfn}

\begin{rmk}
\label{prp:SupClosure}
Let $S$ be a \CuSgp, and let $T\subseteq S$ be a subset.
Then $(\overline{T}^{(\alpha)})_\alpha$ is an increasing family of subsets of $S$, which therefore stabilizes eventually.
Thus, there exists an ordinal $\alpha_0$ such that $\overline{T}^{(\alpha)}=\overline{T}^{(\alpha_0)}$ for all $\alpha\geq\alpha_0$.
Then $\overline{T}^{\txtSup}=\overline{T}^{(\alpha_0)}$, and we get $\overline{\overline{T}^{\txtSup}}^{\txtSeq}=\overline{T}^{\txtSup}$.
It follows that $\overline{T}^{\txtSup}$ is sup-closed, as expected.

We also note that $T$ is sup-closed if and only if $T=\overline{T}^{\txtSeq}$.
\end{rmk}

\begin{dfn}
\label{dfn:Derived}
Let $S$ be a \CuSgp, and let $T\subseteq S$ be a subset.
We set
\[
T'
:= \left\{ \sup_n x_n \in S : (x_n)_n \text{ is a $\ll$-increasing sequence in } T \right\}.
\]
\end{dfn}

\begin{rmk}
\label{rmk:Derived}
Given a topological space $X$ and a subset $Y\subseteq X$, the \emph{derived set} of~$Y$, denoted by $Y'$, is defined as the set of limit points of $Y$.

Let $S$ be a \CuSgp{} and let $T$ be a subset of $S$. 
We think of suprema of $\ll$-increasing sequences of elements in $T$ as the limit points of $T$. 
Therefore, one may view $T'$ as the \emph{derived set} of $T$.
Further, the derived subsets of a \CuSgp{} satisfy the following properties, which are analogs of well known facts satisfied by the  derived subsets of a topological space:
\begin{enumerate}[(i)]
\item 
If $x\in T'$ and if $x$ is not compact (that is, $x\not\ll x$), then $x$ also belongs to $(T-\{ x\})'$.
\item 
We have $(T\cup H)' = T'\cup H'$.
\item 
If $T\subseteq H$, then $T'\subseteq H'$.
\end{enumerate}

To push the previous analogy even further, recall that a subset $Y$ of a topological space is said to be \emph{perfect} if $Y=Y'$.
\autoref{prp:CharSubCu3} below shows that we may think of sub-\CuSgp{s} as the perfect submonoids of a \CuSgp.
\end{rmk}

\begin{lma}
\label{prp:DerivedSupClosed}
Let $S$ be a \CuSgp, and let $T\subseteq S$ be a submonoid.
Then $T'$ is a sup-closed submonoid of $S$. 
\end{lma}
\begin{proof}
Using that the way-below relation is additive and that $0$ is way-below itself, it follows that $T'$ is a submonoid.
It remains to verify that $T'$ is closed under suprema of increasing sequences.

Let $(x_n)_n$ be an increasing sequence in $T'$.
For each~$n$, by definition of $T'$, there exists a $\ll$-increasing sequence $(x_{n,k})_k$ with supremum $x_n$.
Set $k(0):=0$.
Then $x_{0,k(0)+1}\ll x_0\leq x_1$.
Choose $k(1)\in\NN$ such that $x_{0,k(0)+1}\ll x_{1,k(1)}$.
Using that $x_{0,k(0)+2}$ and $x_{1,k(1)+2}$ are way-below $x_2$, we can choose $k(2)\in\NN$ such that $x_{0,k(0)+2},x_{1,k(1)+2}\ll x_{2,k(2)}$.
We inductively choose indices $k(n)\in\NN$ such that
\[
x_{0,k(0)+n},x_{1,k(1)+n},\ldots,x_{n-1,k(n-1)+n} \ll x_{n,k(n)}.
\]
For each $n\in\NN$ set $y_n:=x_{n,k(n)}$.
Then $(y_n)_n$ is a $\ll$-increasing sequence with $\sup_n y_n=\sup_n x_n$, and consequently $\sup_n x_n$ belongs to $T'$.
\end{proof}

\begin{prp}
\label{prp:CharSubCu3}
Let $S$ be a \CuSgp.
Then a submonoid $T\subseteq S$ is a sub-\CuSgp{} if and only if $T=T'$.
\end{prp}
\begin{proof}
The forward implication is clear.
To show the converse, assume that $T\subseteq S$ is a submonoid satisfying $T=T'$.
By \autoref{prp:DerivedSupClosed}, $T$ is sup-closed.
Hence, we can apply \autoref{prp:CharSubCu1} to deduce that $T$ is a sub-\CuSgp.
\end{proof}

The next result recovers \cite[Lemma~5.3.17]{AntPerThi18TensorProdCu}.

\begin{cor}
\label{prp:TinDerivedT}
Let $S$ be a \CuSgp, and let $T\subseteq S$ be a submonoid such that every element in $T$ is the supremum of a $\ll$-increasing sequence in~$T$.
Then $T'=\overline{T}^{\txtSeq}=\overline{T}^{\txtSup}$, which is a sub-\CuSgp{} of $S$.
\end{cor}
\begin{proof}
The inclusions $T'\subseteq\overline{T}^{\txtSeq}\subseteq\overline{T}^{\txtSup}$ hold in general.
By assumption, we have $T\subseteq T'$.
Using \autoref{prp:DerivedSupClosed} at the second step, we get
\[
\overline{T}^{\txtSup} \subseteq \overline{T'}^{\txtSup} = T'.
\]

Since $T'$ is sup-closed, we have $T''\subseteq T'$.
On the other hand, using again that $T\subseteq T'$, we have $T'\subseteq T''$.
Thus, $T'=T''$, which by \autoref{prp:CharSubCu3} implies that $T'\subseteq S$ is a sub-\CuSgp.
\end{proof}

Let $\alpha$ be an ordinal number. 
Continuing with the analogy from \autoref{rmk:Derived}, we now define what may be seen as the $\Cu$-counterpart of the $\alpha$-th Cantor-Bendixson derivative.

\begin{dfn}
\label{dfn:Delta}
Let $S$ be a \CuSgp, and let $T\subseteq S$ be a submonoid.
We define $T^{(\alpha)}$ for every ordinal $\alpha$ by setting $T^{(0)} := T$, $T^{(1)} := T'$, and by using (transfinite) induction:
\begin{align*}
T^{(\alpha+1)} &:= \left(T^{(\alpha)}\right)', \\
T^{(\lambda)} &:= \bigcap_{\alpha<\lambda}T^{(\alpha)},\quad\text{ if $\lambda$ is a limit ordinal}.
\end{align*}

We set
\[
\delta(T) := \bigcap_{\alpha\geq 1} T^{(\alpha)}.
\]
\end{dfn}

\begin{thm}
\label{prp:DeltaIsSupCu}
Let $S$ be a \CuSgp, and let $T\subseteq S$ be a submonoid.
Then $\delta(T)\subseteq S$ is a sub-\CuSgp.

We always have $\delta(T)\subseteq\overline{T}^{\txtSup}$.
Thus, if $T$ is sup-closed, then $\delta(T)\subseteq T$.
\end{thm}
\begin{proof}
Using transfinite induction, we prove that $T^{(\alpha)}$ is a sup-closed submonoid for each ordinal $\alpha\geq 1$.
For $\alpha=1$ and the successor case, this follows from \autoref{prp:DerivedSupClosed}.
The limit case follows directly from the definition.

Thus, $\delta(T)$ is a submonoid.
Further, we deduce that the $T^{(\alpha)}$, for $\alpha\geq 1$, form a decreasing family of submonoids, which therefore stabilizes.
Hence, there exists $\alpha\geq 1$ such that $\delta(T)=T^{(\alpha)}$.
It follows that
\[
\delta(T) = T^{(\alpha)} = T^{(\alpha+1)} = \delta(T)',
\]
which by \autoref{prp:CharSubCu3} implies that $\delta(T)$ is a sub-\CuSgp.

It is clear from the definition that $T'\subseteq\overline{T}^{\txtSeq}$, which shows that $\delta(T)\subseteq\overline{T}^{\txtSup}$.
\end{proof}

Let $S$ be a \CuSgp.
Let $\mathcal{P}$ be the collection of all subsets of $S$;
let $\mathcal{C}$ be the collection of all sup-closed submonoids of $S$;
and let $\mathcal{S}$ be the collection of sub-\CuSgp{s} of $S$.
We equip each of these collections with the partial order given by inclusion of subsets.

Let $\alpha\colon\mathcal{P}\to\mathcal{C}$ be the map that sends a subset of $S$ to the sup-closure of the submonoid it generates.
Then $\alpha$ is order-preserving.
Further, considering $\alpha$ as a map $\mathcal{P}\to\mathcal{P}$, we see that $\alpha$ is idempotent and satisfies $X\subseteq\alpha(X)$ for every $X\in\mathcal{P}$.

Therefore, $\alpha\colon\mathcal{P}\to\mathcal{P}$ is a closure operator in the sense of
\cite[Definition~0-3.8(ii)]{GieHof+03Domains}.
Using that $\mathcal{P}$ is a complete lattice, it follows that $\mathcal{C}$ is a complete lattice as well, that $\alpha$ preserves arbitrary suprema, and that the inclusion map $\iota\colon\mathcal{C}\to\mathcal{P}$ preserves arbitrary infima.
In particular, given a subset $C\subseteq\mathcal{C}$, the supremum of~$C$ in $\mathcal{C}$ is $\sup_{\mathcal{C}}C=\alpha(\bigcup C)$ and the infimum is $\inf_{\mathcal{C}}C=\bigcap C$.
(The intersection of a family of sup-closed submonoids of $S$ is again a sup-closed submonoid.)

Let $\delta\colon\mathcal{C}\to\mathcal{S}$ be the map that sends $T\in\mathcal{C}$ to $\delta(T)$ as defined in \autoref{dfn:Delta}.
It follows from \autoref{prp:DeltaIsSupCu} that $\delta$ is well-defined and order-preserving.
Using also \autoref{prp:CharSubCu3}, we see that $\delta$ as map $\mathcal{C}\to\mathcal{C}$ is idempotent and satisfies $\delta(T)\subseteq T$ for every $T\in\mathcal{C}$.
Thus, $\delta\colon\mathcal{C}\to\mathcal{C}$ is a kernel operator in the sense of
\cite[Definition~0-3.8(iii)]{GieHof+03Domains}.
It follows that $\mathcal{S}$ is a complete lattice, that $\delta$ preserves arbitrary infima, and that the inclusion map $\iota\colon\mathcal{S}\to\mathcal{C}$ preserves arbitrary suprema.

The considered maps are shown in the following diagram:
\[
\xymatrix{
\mathcal{S} \ar@{^{(}->}[r]_{\iota}
& \mathcal{C} \ar@{^{(}->}[r]_{\iota} \ar@/_1pc/[l]_{\delta}
& \mathcal{P}. \ar@/_1pc/[l]_{\alpha}
}
\]

\begin{thm}
\label{prp:LatticeSubCu}
Let $S$ be a \CuSgp.
Then the collection of sub-\CuSgp{s} of $S$ is a complete lattice when ordered by inclusion.

Given a collection $(T_j)_{j\in J}$ of sub-\CuSgp{s} of $S$, their supremum is the sup-closure of the submonoid of $S$ generated by $\bigcup_j T_j$, while their infimum is $\delta(\bigcap_j T_j)$.
\end{thm}

\section{Reduction to countably based \texorpdfstring{$\mathrm{Cu}$}{Cu}-semigroups}
\label{sec:reduction}

In this section, we show that the dimension of a \CuSgp{} is determined by its countably based sub-\CuSgp{s};
see \autoref{prp:CharDimCtbl}. 
We then generalize some results from \cite{ThiVil21arX:DimCu} by dropping the countably based assumption; see Propositions~\ref{prp:Dim0ImplInterpol} and~\ref{prp:DimSoftPart}.

\begin{lma}
\label{prp:genSubCu}
Let $S$ be a \CuSgp, and let $T_0\subseteq S$ be a countable subset.
Then there exists a countably based sub-\CuSgp{} $T\subseteq S$ such that $T_0\subseteq T$.
\end{lma}
\begin{proof}
We may assume that $T_0$ is a submonoid.
For each $x\in T_0$ choose a $\ll$-increasing sequence in $S$ with supremum $x$, and let $T_1$ be the submonoid of $S$ generated by $T_0$ and the elements in each of the chosen sequences.
Repeating this process, we successively obtain an increasing sequence $(T_k)_k$ of countable submonoids of $S$ such that for each $k\in\NN$ and $x\in T_k$ there exists a $\ll$-increasing sequence in $T_{k+1}$ with supremum $x$.
Then $T_\infty:=\bigcup_k T_k$ is a countable submonoid of $S$ such that every element in $T_\infty$ is the supremum of a $\ll$-increasing sequence in~$T_\infty$.
Set $T:=\overline{T_\infty}^{\txtSeq}$.
By \autoref{prp:TinDerivedT}, $T$ is a sub-\CuSgp{} of $S$.
It is straightforward to verify that $T_\infty$ is a countable basis for $T$.
\end{proof}

\begin{pgr}
Given a \CuSgp{} $S$, we let $\SubCtbl(S)$ denote the collection of countably based sub-\CuSgp{s} of $S$.
If $\mathcal{T}\subseteq\SubCtbl(S)$ is a countable, directed family, then $\bigcup \mathcal{T}$ is a submonoid of $S$ such that every element is the supremum of a $\ll$-increasing sequence in $\bigcup \mathcal{T}$, whence it follows from \autoref{prp:TinDerivedT} that the sup-closure $\overline{\bigcup \mathcal{T}}^{\txtSup}$ is a (countably based) sub-\CuSgp.
Note that $\overline{\bigcup \mathcal{T}}^{\txtSup}$ is the supremum of $\mathcal{T}$ in the complete lattice of sub-\CuSgp{s};
see \autoref{prp:LatticeSubCu}.

A collection $\mathcal{R}\subseteq\SubCtbl(S)$ is said to be \emph{$\sigma$-complete} if $\overline{\bigcup \mathcal{T}}^{\txtSup}$ belongs to $\mathcal{R}$ for every countable, directed subset $\mathcal{T}\subseteq\mathcal{R}$.
Further, $\mathcal{R}\subseteq\SubCtbl(S)$ is said to be \emph{cofinal} if for every $T_0\in\SubCtbl(S)$ there is $T\in\mathcal{R}$ satisfying $T_0\subseteq T$.

We say that a property $\mathcal{P}$ of \CuSgp{s} satisfies the \emph{L{\"o}wenheim-Skolem condition} if for every \CuSgp{} $S$ satisfying $\mathcal{P}$, there exists a $\sigma$-complete, cofinal subcollection $\mathcal{R}\subseteq\SubCtbl(S)$ such that every $R\in\mathcal{R}$ satisfies $\mathcal{P}$.
In Propositions~\ref{prp:LS-O5O6O7}, \ref{prp:LS-SimpleCanc} and~\ref{prp:LS-Dim} below, we show that \axiomO{5}, \axiomO{6}, \axiomO{7}, simplicity, weak cancellation and `$\dim(\freeVar)\leq n$' (for fixed $n\in\NN$) each satisfy the L{\"o}wenheim-Skolem condition.
\end{pgr}

\begin{prp}
\label{prp:LS-O5O6O7}
Given a \CuSgp{} $S$ satisfying \axiomO{5} (satisfying \axiomO{6}, satisfying \axiomO{7}), the countably based sub-\CuSgp{s}  satisfying \axiomO{5} (satisfying \axiomO{6}, satisfying \axiomO{7}) form a $\sigma$-complete, cofinal subset of $\SubCtbl(S)$.

In particular, the properties \axiomO{5}, \axiomO{6} and \axiomO{7} each satisfy the L{\"o}wenheim-Skolem condition.
\end{prp}
\begin{proof}
Let $S$ be a \CuSgp{} satisfying \axiomO{5}.
Set
\[
\mathcal{R} := \big\{ R\in\SubCtbl(S) : R \text{ satisfies \axiomO{5}} \big\}.
\]

To show that $\mathcal{R}$ is $\sigma$-complete, let $\mathcal{T}\subseteq\mathcal{R}$ be a countable, directed subset.
Then $\overline{\bigcup \mathcal{T}}^{\txtSup}$ is the inductive limit of the system $\mathcal{T}$.
By \cite[Theorem~4.5]{AntPerThi18TensorProdCu} (see also \autoref{prp:LimitProperties}), \axiomO{5} passes to inductive limits, whence $\overline{\bigcup \mathcal{T}}^{\txtSup}$ belongs to $\mathcal{R}$.

To show that $\mathcal{R}$ is cofinal, let $R_0\in\SubCtbl(S)$.
We need to find $R\in\mathcal{R}$ satisfying $R_0\subseteq R$.
Choose a countable basis $B_0\subseteq R_0$.

We will inductively choose an increasing sequence $(R_n)_n$ in $\SubCtbl(S)$ and a countable basis $B_n\subseteq R_n$ such that for each $n$ the following holds:

\emph{
For every $x',x,y',y,z',z\in B_n$ satisfying $x+y\ll z'\ll z$, $x'\ll x$ and $y'\ll y$, there exists $c\in B_{n+1}$ such that $x'+c\ll z$, $z'\ll x+c$ and $y'\ll c$.
}

We have already obtained $R_0$ and $B_0$.
Let $n\in\NN$ and assume that we have chosen $R_k$ and $B_k$ for all $k\leq n$.
Consider the countable set
\[
I_n := \big\{ (x',x,y',y,z',z)\in B_n^6 : x+y\ll z'\ll z, x'\ll x, y'\ll y \big\}.
\]
Since $S$ satisfies \axiomO{5}, we obtain for each $i=(x',x,y',y,z',z)\in I_n$ an element $c_i\in S$ such that $x'+c_i\ll z$, $z'\ll x+c_i$ and $y'\ll c_i$.
Applying \autoref{prp:genSubCu}, we obtain $R_{n+1}\in\SubCtbl(S)$ containing $B_n\cup\{c_i:i\in I_n\}$.
Since $B_n$ is a basis for $R_n$, we have $R_n\subseteq R_{n+1}$.
Choose a countable basis $B_{n+1}$ for $R_{n+1}$ that contains $B_n$ and each $c_i$ for $i\in I_n$.
This completes the induction step.

Now set $R:=\overline{\bigcup_n R_n}^{\txtSup}$ and $B:=\bigcup_n B_n$.
Then $R$ is a sub-\CuSgp{} of $S$ containing $R_0$.
Further, $B$ is a countable basis of $R$.
By construction, $B$ satisfies the condition from \autoref{prp:CharO5}, showing that $R$ satisfies \axiomO{5}.
Thus, $R$ belongs to $\mathcal{R}$, as desired.

A similar argument, using \autoref{prp:LimitProperties} (to show $\sigma$-completeness) and Propositions \ref{prp:CharO6} and~\ref{prp:CharO7} (to show cofinality) proves that \axiomO{6} and \axiomO{7} satisfy the L{\"o}wenheim-Skolem condition.
\end{proof}

A \CuSgp{} $S$ is \emph{simple} if for all $x,y\in S$ with $y\neq 0$ we have $x\leq\infty y$.

\begin{prp}
\label{prp:LS-SimpleCanc}
Given a simple (weakly cancellative) \CuSgp{} $S$, every sub-\CuSgp{} of $S$ is simple (weakly cancellative).

In particular, simplicity and weak cancellation each satisfy the L{\"o}wenheim-Sko\-lem condition.
\end{prp}
\begin{proof}
Let $S$ be a simple \CuSgp, and let $T\subseteq S$ be a sub-\CuSgp.
To verify that $T$ is simple, let $x,y\in T$ with $y\neq 0$.
Since $S$ is simple, we have $x\leq\infty y$ in $S$, and since the inclusion $T\to S$ is an order-embedding, we get $x\leq\infty y$ in $T$.

Let us now assume that $S$ is a weakly cancellative \CuSgp, and let $T\subseteq S$ be a sub-\CuSgp.
To show that $T$ is weakly cancellative, let $x,y,z\in T$ satisfy $x+z\ll y+z$.
This implies that $x\ll y$ in $S$, and thus $x\ll y$ in $T$.
\end{proof}

\begin{lma}
\label{prp:subCuDim}
Let $S$ be a \CuSgp, and let $T_0\subseteq S$ be a countably based sub-\CuSgp.
Then there exists a countably based sub-\CuSgp{} $T\subseteq S$ such that $T_0\subseteq T$ and $\dim(T)\leq\dim(S)$.
\end{lma}
\begin{proof}
Set $n:=\dim(S)$.
If $n=\infty$, then $T:=T_0$ has the desired properties.
Thus, we may assume that $n$ is finite.

\textbf{Claim:} \emph{Let $P\subseteq S$ be a countably based sub-\CuSgp.
Then there exists a countably based sub-\CuSgp{} $Q\subseteq S$ satisfying $P\subseteq Q$ and with the following property:
Whenever $x'\ll x\ll y_1+\ldots+y_r$ in $P$, then there exist $z_{j,k}\in Q$ for $j=1,\ldots,r$ and $k=0,\ldots,n$ satisfying (i)-(iii) from \autoref{dfn:dim}.
}

To prove the claim, choose a countable basis $B\subseteq P$.
For each $r\geq 1$, consider the countable set
\[
I_r := \big\{ (x',x,y_1,\ldots,y_r)\in B^{r+2} : x'\ll x\ll y_1+\ldots+y_r \big\}.
\]
For each $i=(x',x,y_1,\ldots,y_r)\in I_r$, we apply $\dim(S)\leq n$ for $x'\ll x\ll y_1+\ldots+y_r$ to obtain elements $z_{i,j,k}\in S$ for $j=1,\ldots,r$ and $k=0,\ldots,n$ satisfying (i)-(iii) from \autoref{dfn:dim}.
Applying \autoref{prp:genSubCu}, we obtain a countably based sub-\CuSgp{} $Q\subseteq S$ that contains $B$ and each $z_{i,j,k}$ for $r\geq 1$, $i\in I_r$, $j=1,\ldots,r$, and $k=0,\ldots,n$.
Since $B$ is a basis for $P$, we have $P\subseteq Q$.

To verify that $Q$ has the claimed property, let $x'\ll x\ll y_1+\ldots+y_r$ in $P$.
Using that $B$ is a basis, we can choose $c',c,d_1,\ldots,d_r\in B$ such that
\[
x'\ll c'\ll c\ll x\ll d_1+\ldots+d_r, \quad d_1\ll y_1, \quad \ldots,\quad d_r\ll y_r.
\]
Then $i:=(c',c,d_1,\ldots,d_r)$ belongs to $I_r$.
By construction, $Q$ contains the elements $z_{i,j,k}$, which satisfy (i)-(iii) from \autoref{dfn:dim} for $c'\ll c\ll d_1+\ldots+d_r$, and it is easy to see that these same elements satisfy (i)-(iii) from \autoref{dfn:dim} for $x'\ll x\ll y_1+\ldots+y_r$.
This proves the claim.

Now, we successively apply the claim to obtain an increasing sequence $(T_k)_k$ of countably based sub-\CuSgp{s} $T_k\subseteq S$ such that for every $k\in\NN$ and $x'\ll x\ll y_1+\ldots+y_r$ in $T_k$ there exist $z_{j,k}\in T_{k+1}$ for $j=1,\ldots,r$ and $k=0,\ldots,n$ satisfying (i)-(iii) from \autoref{dfn:dim}.

Let $T_\infty:=\bigcup_k T_k$, which by construction is a submonoid of $S$ such that every element in $T_\infty$ is the supremum of a $\ll$-increasing sequence in $T_\infty$.
Set $T:=\overline{T_\infty}^{\txtSeq}$.
By \autoref{prp:TinDerivedT}, $T$ is a sub-\CuSgp{} of $S$ satisfying $T_0\subseteq T$.
It is now straightforward to verify that $T$ is countably based and satisfies $\dim(T)\leq n$.
\end{proof}

\begin{prp}
\label{prp:LS-Dim}
Let $n\in\NN$.
Given a \CuSgp{} $S$ satisfying $\dim(S)\leq n$, the countably based sub-\CuSgp{s} $T\subseteq S$ satisfying $\dim(T)\leq n$ form a $\sigma$-complete, cofinal subset of $\SubCtbl(S)$.

In particular, the property of \CuSgp{s} of having dimension at most $n$ satisfies the L{\"o}wenheim-Sko\-lem condition.
\end{prp}
\begin{proof}
Let $\mathcal{R}$ be the collection of sub-\CuSgp{s} $T\subseteq S$ satisfying $\dim(T)\leq n$.
By \autoref{prp:Permanence}, the property of having dimension at most $n$ passes to inductive limits, which shows that $\mathcal{R}$ is $\sigma$-complete.
Further, $\mathcal{R}$ is cofinal by \autoref{prp:subCuDim}.
\end{proof}

\begin{thm}
\label{prp:CharDimCtbl}
Let $S$ be a \CuSgp, and let $n\in\NN$.
Then the following are equivalent:
\begin{enumerate}
\item
$\dim(S)\leq n$;
\item
every countable subset of $S$ is contained in a countably based sub-\CuSgp{} $T\subseteq S$ satisfying $\dim(T)\leq n$;
\item
every finite subset of $S$ is contained in a sub-\CuSgp{} $T\subseteq S$ satisfying $\dim(T)\leq n$.
\end{enumerate}
\end{thm}
\begin{proof}
It follows from Lemmas \ref{prp:genSubCu} and \ref{prp:subCuDim} that~(1) implies~(2).
It is clear that~(2) implies~(3).
To show that~(3) implies~(1), let $\mathcal{T}$ be the collection of sub-\CuSgp{s} with dimension at most $n$.
For each $T\in\mathcal{T}$, let $\iota_T\colon T\to S$ denote the inclusion map.
It follows from the assumption that the family $(T,\iota_T)_{T\in\mathcal{T}}$ approximates $S$.
Hence, we have $\dim(S)\leq n$ by \autoref{prp:ApproxDim}.
\end{proof}

As an application of the methods developed in this section, we can remove the assumption of being countably based in several results from \cite{ThiVil21arX:DimCu}. We first generalize \cite[Proposition~7.14]{ThiVil21arX:DimCu}.

\begin{prp}
\label{prp:Dim0ImplInterpol}
Let $S$ be a zero-dimensional, simple, weakly cancellative \CuSgp{} satisfying \axiomO{5}.
Then $S$ has the Riesz interpolation property.
If we additionally assume that $S$ is nonelementary, then $S$ is almost divisible.
\end{prp}
\begin{proof}
If $S$ is elementary, then $S$ is isomorphic to $\{0,1,\ldots,\infty\}$, or to $\{0,1,\ldots,n\}$ for some $n\in\NN$; see \cite[Proposition~5.1.19]{AntPerThi18TensorProdCu}.
In either case, $S$ has the Riesz interpolation property.
So we may assume from now on that $S$ is nonelementary.
This allows us to choose a sequence $(s_n)_n$ in $S$ with $s_0>s_1>\ldots$.

Let $\mathcal{R}_{\mathrm{O5}}$, $\mathcal{R}_{\mathrm{simple}}$, $\mathcal{R}_{\mathrm{canc}}$, and $\mathcal{R}_{\mathrm{dim0}}$ be the collections of countably generated sub-\CuSgp{s} of $S$ that satisfy \axiomO{5}, or that are simple, weakly cancellative, or zero-dimensional, respectively.
By Propositions~\ref{prp:LS-O5O6O7}, \ref{prp:LS-SimpleCanc}, and~\ref{prp:LS-Dim}, each of these collections are $\sigma$-complete and cofinal.
Set $\mathcal{R}:=\mathcal{R}_{\mathrm{O5}}\cap\mathcal{R}_{\mathrm{simple}}\cap\mathcal{R}_{\mathrm{canc}}\cap\mathcal{R}_{\mathrm{dim0}}$.
Then $\mathcal{R}$ is $\sigma$-complete and cofinal.

To verify that $S$ has the Riesz interpolation property, let $x_0,x_1,y_0,y_1\in S$ satisfy $x_j\leq y_k$ for all $j,k\in\{0,1\}$.
We need to find $z\in S$ such that $x_j\leq z\leq y_k$ for all $j,k\in\{0,1\}$.
Using \autoref{prp:genSubCu} and that $\mathcal{R}$ is cofinal, we obtain $R\in\mathcal{R}$ containing $x_0,x_1,y_0,y_1$ and containing $s_0,s_1,\ldots$, which forces $R$ to be nonelementary.

Note that $R$ is a zero-dimensional, countably based, simple, weakly cancellative, nonelementary \CuSgp{} satisfying \axiomO{5}.
By \cite[Proposition~7.14]{ThiVil21arX:DimCu}, $R$ has the Riesz interpolation property.
We therefore obtain $z$ with the desired properties in $R\subseteq S$.

To verify that $S$ is almost divisible, let $n\geq 1$, and let $x',x\in S$ satisfy $x'\ll x$.
We need to find $z\in S$ such that $nz\ll x$ and $x'\ll(n+1)z$.
As above, we obtain $R\in\mathcal{R}$ containing $x',x,s_0,s_1,\ldots$.
By \cite[Proposition~7.14]{ThiVil21arX:DimCu}, $R$ is almost divisible, which allows us to find $z$ with the desired properties in $R$.
\end{proof}

The next result generalizes \cite[Proposition~3.17]{ThiVil21arX:DimCu}.
Recall that an element $x$ in a \CuSgp{} $S$ is said to be \emph{soft} if for every $x'\ll x$ there exists $k\in\NN$ such that $(k+1)x'\ll kx$. 
The set of soft elements, denoted by $S_{\rm{soft}}$, is a sub-\CuSgp{} of~$S$ satisfying \axiomO{5} and \axiomO{6} whenever $S$ is simple, weakly cancellative and satisfies \axiomO{5} and \axiomO{6}; 
see \cite[Proposition~5.3.18]{AntPerThi18TensorProdCu}.

\begin{prp}
\label{prp:DimSoftPart}
Let $S$ be a simple, weakly cancellative \CuSgp{} satisfying \axiomO{5} and \axiomO{6}.
Then
\[
\dim(S_{\rm{soft}})\leq\dim(S)\leq\dim(S_{\rm{soft}})+1.
\]
\end{prp}
\begin{proof}
To prove the first inequality, set $n:=\dim(S)$, which we may assume to be finite.
To verify condition~(3) of \autoref{prp:CharDimCtbl}, let $H$ be a finite subset of $S_{\rm{soft}}$. 

Proceeding as in the proof of \autoref{prp:Dim0ImplInterpol}, and using \autoref{prp:genSubCu} and Propositions \ref{prp:LS-O5O6O7}, \ref{prp:LS-SimpleCanc} and \ref{prp:LS-Dim}, there exists a simple, weakly cancellative, countably based sub-\CuSgp{} $T\subseteq S$ satisfying \axiomO{5} and \axiomO{6} with $\dim (T)\leq n$ and $H\subseteq T$.

It follows from \cite[Proposition~3.17]{ThiVil21arX:DimCu} that $\dim (T_{\rm{soft}})\leq n$.
We note that $T_{\rm{soft}}$ is a sub-\CuSgp{} of $S_{\rm{soft}}$ containing $H$. 
Thus, every finite subset of $S_{\rm{soft}}$ is contained in a sub-\CuSgp{} of dimension at most $n$. 
This shows that $\dim (S_{\rm{soft}})\leq n$ by \autoref{prp:CharDimCtbl}.
 
To prove the second inequality, set $m:=\dim(S_{\rm{soft}})$, which we may assume to be finite.
Let $H$ be a finite subset of $S$. 
Using once again \autoref{prp:genSubCu} and \autoref{prp:LS-O5O6O7}, one finds a countably based sub-\CuSgp{} $T^{(1)}\subseteq S$ satisfying \axiomO{5} and \axiomO{6} with $H\subseteq T^{(1)}$.
 
By \autoref{prp:LS-Dim}, there exists a countably based sub-\CuSgp{} $R^{(1)}$ of $S_{\rm{soft}}$ such that $T^{(1)}_{\rm{soft}}\subset R^{(1)}$ and $\dim(R^{(1)})\leq m$. 
Since $R^{(1)}$ and $T^{(1)}$ are countably based, there exists by \autoref{prp:genSubCu} and \autoref{prp:LS-O5O6O7} a countably based sub-\CuSgp{} $T^{(2)}\subseteq S$ satisfying \axiomO{5} and \axiomO{6} with $R^{(1)},T^{(1)}\subseteq T^{(2)}$.
 
Proceeding in this manner, one obtains an increasing sequence of countably based sub-\CuSgp{s} $T^{(k)}\subseteq S$ satisfying \axiomO{5} and \axiomO{6} and an increasing sequence of sub-\CuSgp{} $R^{(k)}$ of $S_{\rm{soft}}$ with dimension at most $m$ such that
\[
T^{(k)}_{\rm{soft}}\subseteq R^{(k)}\subseteq T^{(k+1)}.
\]

Set $T:=\overline{\bigcup_k T^{(k)}}^{\txtSup}$ and $R:=\overline{\bigcup_k R^{(k)}}^{\txtSup}$.
Then $T\subseteq S$ is a countably based sub-\CuSgp.
Since each $T^{(k)}$ satisfies \axiomO{5} and \axiomO{6}, it follows from \autoref{prp:LS-O5O6O7} that $T$ also satisfies them. 
Moreover, using that $T$ is a sub-\CuSgp{} of $S$, it follows from \autoref{prp:LS-SimpleCanc} that $T$ is also simple and weakly cancellative.
 
Since $\dim(R^{(k)})\leq m$ for every $k$, we have by \autoref{prp:CharDimCtbl} (or using \autoref{prp:Permanence}) that $\dim(R)\leq m$, and it is easy to check that $T_{\rm{soft}}=R$.
Applying \cite[Proposition~3.17]{ThiVil21arX:DimCu}, we get 
\[
\dim (T)\leq \dim(T_{\rm{soft}})+1 =\dim(R)+1\leq m+1.
\]

Thus, every finite subset of $S$ is contained in a sub-\CuSgp{} with dimension at most $m+1$. 
This shows, by \autoref{prp:CharDimCtbl}, that $\dim(S)\leq m+1$, as desired.
\end{proof}

\section{Dimension of the Cuntz semigroup as a noncommutative dimension theory}
\label{sec:DimNonComThe}

In this section, we show that associating to a \ca{} the dimension of its Cuntz semigroup satisfies the L{\"o}wenheim-Skolem condition;
see \autoref{prp:SepSubCaDim}.
It follows that this association is a well-behaved invariant that satisfies most of the  axioms of a noncommutative dimension theory in the sense of \cite[Definition~1]{Thi13TopDimTypeI};
see \autoref{pgr:DimThy}.
It remains open if the dimension of the Cuntz semigroup is compatible with minimal unitizations;
see \autoref{qst:DimCuIsDimThy}.

If $B\subseteq A$ is a sub-\ca{}, then the inclusion map $B\to A$ induces a \CuMor{} $\Cu(B)\to\Cu(A)$ which in general is not an order-embedding.
Thus, the Cuntz semigroup of a sub-\ca{} is not necessarily a sub-\CuSgp.
However, the next results shows that there are sufficiently many separable sub-\ca{s} whose Cuntz semigroups are sub-\CuSgp{s}.

Given a \ca{} $A$, we let $\SubSep(A)$ denote the collection of separable sub-\ca{s} of $A$.
See \cite[Paragraph~3.1]{Thi20arX:grSubhom} for details.

\begin{prp}
\label{prp:SubCaWithSubCu}
Let $A$ be a \ca.
Then
\[
\mathcal{S} := \big\{B\in\SubSep(A) : \Cu(B)\to\Cu(A) \text{ is an order-embedding} \big\}
\]
is $\sigma$-complete and cofinal.
Each $B\in\mathcal{S}$ induces a countably based sub-\CuSgp{} $\Cu(B)\subseteq\Cu(A)$.
Let $\alpha\colon\mathcal{S}\to\SubCtbl(\Cu(A))$ be the map that sends $B\in\mathcal{S}$ to the sub-\CuSgp{} $\Cu(B)\subseteq\Cu(A)$.
Then $\alpha$ preserves the order and the suprema of countable directed subsets, and the image of $\alpha$ is a cofinal subset of $\SubCtbl(\Cu(A))$.
\end{prp}
\begin{proof}
To prove that $\mathcal{S}$ is $\sigma$-complete, let $\mathcal{T}\subseteq\mathcal{S}$ be a countable, directed subfamily.
Set $D:=\overline{\bigcup\mathcal{T}}$.
We need to verify $D\in\mathcal{S}$.
Let $\varphi_{A,D}\colon \Cu(D)\to \Cu(A)$ denote the \CuMor{} induced by the inclusion map $D\to A$.
Similarly, we define $\varphi_{D,B}\colon \Cu(B)\to\Cu(D)$ and $\varphi_{A,B}\colon \Cu(B)\to\Cu(A)$ for each $B\in\mathcal{T}$.

To verify that $\varphi_{A,D}$ is an order-embedding, let $x,y\in\Cu(D)$ satisfy
\[
\varphi_{A,D}(x) \leq \varphi_{A,D}(y).
\]
Let $x', x''\in\Cu(D)$ be such that $x'\ll x''\ll x$.
Then $\varphi_{A,D}(x'')\ll\varphi_{A,D}(x)\leq\varphi_{A,D}(y)$, which allows us to choose $y'\in\Cu(D)$ with
\[
\varphi_{A,D}(x'') \ll \varphi_{A,D}(y'), \andSep y'\ll y.
\]

Using that $D\cong\varinjlim_{B\in\mathcal{T}}B$, we have $\Cu(D)\cong\varinjlim_{B\in\mathcal{T}}\Cu(B)$ by \cite[Corollary~3.2.9]{AntPerThi18TensorProdCu}.
Applying (L2) from \cite[Paragraph~3.8]{ThiVil21arX:DimCu} (see also the proof of \autoref{prp:ApproxLimit}), we obtain $B\in\mathcal{T}$ and $c,d\in\Cu(B)$ such that
\[
x' \ll \varphi_{D,B}(c) \ll x'', \andSep
y' \ll \varphi_{D,B}(d) \ll y.
\]
Then
\[
\varphi_{A,B}(c)
= \varphi_{A,D}( \varphi_{D,B}(c) )
\ll \varphi_{A,D}(x'') 
\ll \varphi_{A,D}(y')
\ll \varphi_{A,D}( \varphi_{D,B}(d) )
= \varphi_{A,B}(d).
\]
Using that $\varphi_{A,B}$ is an order-embedding, we obtain $c\ll d$ in $\Cu(B)$, and thus
\[
x' 
\ll \varphi_{D,B}(c) 
\ll \varphi_{D,B}(d) 
\ll y.
\]
Thus, $x'\ll y$ for every $x'$ way-below $x$, which implies $x\leq y$.

To verify that $\mathcal{S}$ is cofinal, let $B_0\in\SubSep(A)$.
By \cite[Theorem~2.6.2]{FarHarLupRobTikVigWin16arX:ModelThy}, there exists $B\in\SubSep(A)$ such that $B_0\subseteq B$ and such that $B$ is an elementary submodel of $A$. 
By \cite[Lemma~8.1.3]{FarHarLupRobTikVigWin16arX:ModelThy}, $\Cu(B)\to\Cu(A)$ is an order-embedding.
Thus, $B$ belongs to $\mathcal{S}$, as desired.

Given $B\in\mathcal{S}$, it follows from \autoref{prp:CharSubCu2} that we can identify $\Cu(B)$ with a sub-\CuSgp{} of $\Cu(A)$.
Since $B$ is separable, $\Cu(B)$ is countably based.
It is then straightforward to see that the map $\alpha$ is order-preserving.
Next, let $\mathcal{T}\subseteq\mathcal{S}$ be a countable, directed subset, and set $D:=\overline{\bigcup\mathcal{T}}$.
We identify $\Cu(D)$ and $\Cu(B)$ (for each $B\in\mathcal{T}$) with sub-\CuSgp{s} of $\Cu(A)$.
Then $(\Cu(B))_{B\in\mathcal{T}}$ is a countable, directed family in $\SubCtbl(\Cu(A))$, with supremum given by $\sup_{B\in\mathcal{T}}\Cu(B)=\overline{\bigcup_{B\in\mathcal{T}}\Cu(B)}^{\txtSup}$;
see \autoref{prp:LatticeSubCu}.
Since $\Cu(B)$ is contained in $\Cu(D)$ for each $B\in\mathcal{T}$, we have
\[
\overline{\bigcup_{B\in\mathcal{T}}\Cu(B)}^{\txtSup}\subseteq\Cu(D).
\]
The other inclusion follows using that $\Cu(D)$ is the inductive limit of $(\Cu(B))_{B\in\mathcal{T}}$.
This shows that $\alpha$ preserves suprema of countable directed subsets.

Finally, to show that the image of $\alpha$ is cofinal, let $T\in\SubCtbl(\Cu(A))$.
Choose a countable basis $D\subseteq T$.
For each $x\in D$ choose $a_x\in (A\otimes\KK)_+$ with $x=[a_x]$.
We can then choose a separable sub-\ca{} $B_0\subseteq A$ such that each $a_x$ is contained in $B_0\otimes\KK\subseteq A\otimes\KK$.
Using that $\mathcal{S}$ is cofinal, we obtain $B\in\mathcal{S}$ containing $B_0$.
Then the sub-\CuSgp{} $\Cu(B)\subseteq\Cu(A)$ contains each $x\in D$, which implies $T\subseteq\Cu(B)$ as required.
\end{proof}

\begin{thm}
\label{prp:SepSubCaDim}
Let $n\in\NN$, and let $A$ be a \ca{} satisfying $\dim(\Cu(A))\leq n$.
Then 
\[
\mathcal{S}:=\big\{ B\in\SubSep(A) : \Cu(B)\to\Cu(A) \text{ order-embedding}, \dim(\Cu(B))\leq n \big\}
\]
is $\sigma$-complete and cofinal.

In particular, for every separable sub-\ca{} $B_0\subseteq A$ there exists a separable sub-\ca{} $B\subseteq A$ such that $B_0\subseteq B$ and $\dim(\Cu(B))\leq n$.
\end{thm}
\begin{proof}
Set
\begin{align*}
\mathcal{S}_0 &:= \big\{B\in\SubSep(A) : \Cu(B)\to\Cu(A) \text{ is an order-embedding} \big\}, \\
\mathcal{T} &:= \big\{T\in\SubCtbl(\Cu(A)) : \dim(T)\leq n \big\}.
\end{align*}

By \autoref{prp:SubCaWithSubCu}, $\mathcal{S}_0$ is a $\sigma$-complete and cofinal subfamily of $\SubSep(A)$.
Similarly, by \autoref{prp:LS-Dim}, $\mathcal{T}$ is a $\sigma$-complete and cofinal subset of $\SubCtbl(\Cu(B))$.
Let $\alpha\colon\mathcal{S}_0\to\SubCtbl(\Cu(A))$ be the map that sends $B\in\mathcal{S}_0$ to the sub-\CuSgp{} $\Cu(B)\subseteq\Cu(A)$, as in \autoref{prp:SubCaWithSubCu}.
Then
\[
\mathcal{S} = \big\{ B\in \mathcal{S}_0 : \alpha(B)\in \mathcal{T} \big\}.
\]

Using that $\mathcal{S}_0$ and $\mathcal{T}$ are $\sigma$-complete, and using that $\alpha$ preserves suprema of countable, directed sets, it follows that $\mathcal{S}$ is $\sigma$-complete.
To show that $\mathcal{S}$ is cofinal, let $B_0\in\SubSep(A)$.
Using that $\mathcal{T}$ is cofinal, we obtain $T_0\in\mathcal{T}$ such that $\alpha(B_0)\subseteq T_0$.
Using that the image of $\alpha$ is cofinal, we find $B_1\in\mathcal{S}_0$ such that $T_0\subseteq\alpha(B_1)$.
Continuing successively, we obtain an increasing sequence $(T_k)_{k\in\NN}$ in $\mathcal{T}$ and an increasing sequence $(B_k)_{k\geq 1}$ in $\mathcal{S}_0$ such that
\[
\alpha(B_0)\subseteq T_0 \subseteq \alpha(B_1)\subseteq T_1 \subseteq \alpha(B_2)\subseteq T_2 \subseteq \ldots.
\]
Set $B:=\overline{\bigcup_k B_k}$ and $T:=\overline{\bigcup_k T_k}^{\txtSup}$.
Then $B_0\subseteq B$, $B\in\mathcal{S}_0$ and $T\in\mathcal{T}$.
Using that $\alpha$ preserves suprema of countable, directed sets, we get $\alpha(B)=T$, and thus $B\in\mathcal{S}$, as desired.
\end{proof}

\begin{cor}
\label{prp:CharDimCa}
Let $A$ be a \ca{}, and let $n\in\NN$.
Then $\dim(\Cu(A))\leq n$ if and only if every finite (or countable) subset of $A$ is contained in a separable sub-\ca{} $B\subseteq A$ satisfying $\dim(\Cu(B))\leq n$.
\end{cor}
\begin{proof}
The forward implication is \autoref{prp:SepSubCaDim}, and the backward implication follows from \autoref{prp:DimApproxCAlg}.
\end{proof}

\begin{pgr}
\label{pgr:DimThy}
Following \cite[Definition~1]{Thi13TopDimTypeI}, we say that an assignment that to each \ca{} $A$ associates a number (the dimension) $d(A)\in\{0,1,2,\ldots,\infty\}$ is a \emph{(noncommutative) dimension theory} if the following conditions are satisfied:
\begin{enumerate}
\item[\axiomD{1}]
$d(I)\leq d(A)$ whenever $I\subseteq A$ is an ideal in a \ca{} $A$;
\item[\axiomD{2}]
$d(A/I)\leq d(A)$ whenever $I\subseteq A$ is an ideal in a \ca{} $A$;
\item[\axiomD{3}]
$d(A\oplus B)=\max\{d(A),d(B)\}$, whenever $A$ and $B$ are \ca{s};
\item[\axiomD{4}]
$d(\widetilde{A})=d(A)$ for every \ca{} $A$;
\item[\axiomD{5}]
If $n\in\NN$ and if $A$ is a \ca{} that is approximated by sub-\ca{s} $A_\lambda\subseteq A$ with $d(A_\lambda)\leq n$, then $d(A)\leq n$;
\item[\axiomD{6}]
Given a \ca{} $A$ and a separable sub-\ca{} $B_0\subseteq A$, there exists a separable sub-\ca{} $B\subseteq A$ such that $B_0\subseteq B$ and $d(B)\leq d(A)$.
\end{enumerate}

Assigning to a \ca{} the dimension of its Cuntz semigroup satisfies conditions \axiomD{1}, \axiomD{2} and \axiomD{3} (by Proposition~3.10 of \cite{ThiVil21arX:DimCu}), \axiomD{5} (by \autoref{prp:DimApproxCAlg}) and \axiomD{6} (by \autoref{prp:SepSubCaDim}).
However, \autoref{exa:IncrDimUnitization} shows that \axiomD{4} does not hold.

Using \cite[Proposition~3]{Thi13TopDimTypeI}, one can also see that this assignemnt is in fact \emph{Morita-invariant}, that is, $\dim (\Cu (A))=\dim (\Cu (B))$ whenever $A$ and $B$ are Morita equivalent.
\end{pgr}

\begin{exa}
\label{exa:IncrDimUnitization}
Let $\mathcal{W}$ denote the Jacelon-Razac algebra.
Then
\[
\dim(\Cu(\mathcal{W}))
= 0, \andSep
\dim(\Cu(\widetilde{\mathcal{W}}))=1.
\]
Indeed, we have $\Cu(\mathcal{W})\cong[0,\infty]$, which is easily seen to be zero-dimensional.
(See also \cite[Proposition~3.22]{ThiVil21arX:DimCu}.)

Further, since the nuclear dimension of $\widetilde{\mathcal{W}}$ is $1$, we have $\dim(\Cu(\widetilde{\mathcal{W}}))\leq 1$ by \cite[Theorem~4.1]{ThiVil21arX:DimCu}.
On the other hand, since $\mathcal{W}$ has stable rank one, so does $\widetilde{\mathcal{W}}$ (by definition); 
but $\mathcal{W}$ does not have real rank zero, and hence neither does $\widetilde{\mathcal{W}}$.
Thus, it follows from \cite[Corollary~5.8]{ThiVil21arX:DimCu} that $\Cu(\widetilde{\mathcal{W}})$ is not zero-dimensional.
\end{exa}

\begin{qst}
\label{qst:DimCuIsDimThy}
Let $I\subseteq A$ be an ideal in a unital \ca{} $A$.
Do we have $\dim(\Cu(\widetilde{I}))\leq\dim(\Cu(A))$?
\end{qst}

If the above question has a positive answer, then associating to a \ca{} the dimension of the Cuntz semigroup of its minimal unitization is a noncommutative dimension theory.
Indeed, one can verify that this assignment satisfies \axiomD{2}-\axiomD{6}, and \autoref{qst:DimCuIsDimThy} is asking if \axiomD{1} holds.


\begin{thebibliography}{GHK{\etalchar{+}}03}

\bibitem[APRT18]{AntPerRobThi18arX:CuntzSR1}
\bgroup\scshape{}R.~Antoine\egroup{}, \bgroup\scshape{}F.~Perera\egroup{},
  \bgroup\scshape{}L.~Robert\egroup{}, and \bgroup\scshape{}H.~Thiel\egroup{},
  \ca{s} of stable rank one and their {C}untz semigroups, Duke Math. J. (to
  appear), preprint (arXiv:1809.03984 [math.OA]), 2018.

\bibitem[APRT19]{AntPerRobThi19arX:Edwards}
\bgroup\scshape{}R.~Antoine\egroup{}, \bgroup\scshape{}F.~Perera\egroup{},
  \bgroup\scshape{}L.~Robert\egroup{}, and \bgroup\scshape{}H.~Thiel\egroup{},
  Edwards' condition for quasitraces on \ca{s}, Proc. Roy. Soc. Edinburgh Sect.
  A (to appear), preprint (arXiv:1909.12787 [math.OA]), 2019.

\bibitem[APS11]{AntPerSan11PullbacksCu}
\bgroup\scshape{}R.~Antoine\egroup{}, \bgroup\scshape{}F.~Perera\egroup{}, and
  \bgroup\scshape{}L.~Santiago\egroup{}, Pullbacks, {$C(X)$}-algebras, and
  their {C}untz semigroup,  \emph{J. Funct. Anal.} \textbf{260} (2011),
  2844--2880.

\bibitem[APT18]{AntPerThi18TensorProdCu}
\bgroup\scshape{}R.~Antoine\egroup{}, \bgroup\scshape{}F.~Perera\egroup{}, and
  \bgroup\scshape{}H.~Thiel\egroup{}, Tensor products and regularity properties
  of {C}untz semigroups,  \emph{Mem. Amer. Math. Soc.} \textbf{251} (2018),
  viii+191.

\bibitem[APT20]{AntPerThi20AbsBivariantCu}
\bgroup\scshape{}R.~Antoine\egroup{}, \bgroup\scshape{}F.~Perera\egroup{}, and
  \bgroup\scshape{}H.~Thiel\egroup{}, Abstract bivariant {C}untz semigroups,
  \emph{Int. Math. Res. Not. IMRN} (2020), 5342--5386.

\bibitem[APT11]{AraPerTom11Cu}
\bgroup\scshape{}P.~Ara\egroup{}, \bgroup\scshape{}F.~Perera\egroup{}, and
  \bgroup\scshape{}A.~S. Toms\egroup{}, {$K$}-theory for operator algebras.
  {C}lassification of \ca{s},  in \emph{Aspects of operator algebras and
  applications}, \emph{Contemp. Math.} \textbf{534}, Amer. Math. Soc.,
  Providence, RI, 2011, pp.~1--71.

\bibitem[Bor94]{Bor94HandbookCat1}
\bgroup\scshape{}F.~Borceux\egroup{}, \emph{Handbook of categorical algebra.
  1}, \emph{Encyclopedia of Mathematics and its Applications} \textbf{50},
  Cambridge University Press, Cambridge, 1994, Basic category theory.

\bibitem[CEI08]{CowEllIva08CuInv}
\bgroup\scshape{}K.~T. Coward\egroup{}, \bgroup\scshape{}G.~A.
  Elliott\egroup{}, and \bgroup\scshape{}C.~Ivanescu\egroup{}, The {C}untz
  semigroup as an invariant for \ca{s},  \emph{J.\ Reine Angew.\ Math.}
  \textbf{623} (2008), 161--193.

\bibitem[Cun78]{Cun78DimFct}
\bgroup\scshape{}J.~Cuntz\egroup{}, Dimension functions on simple \ca{s},
  \emph{Math. Ann.} \textbf{233} (1978), 145--153.

\bibitem[FHL{\etalchar{+}}16]{FarHarLupRobTikVigWin16arX:ModelThy}
\bgroup\scshape{}I.~Farah\egroup{}, \bgroup\scshape{}B.~Hart\egroup{},
  \bgroup\scshape{}M.~Lupini\egroup{}, \bgroup\scshape{}L.~Robert\egroup{},
  \bgroup\scshape{}A.~Tikuisis\egroup{}, \bgroup\scshape{}A.~Vignati\egroup{},
  and \bgroup\scshape{}W.~Winter\egroup{}, Model theory of \ca{s}, Mem. Amer.
  Math. Soc. (to appear), preprint (arXiv:1602.08072 [math.OA]), 2016.

\bibitem[GHK{\etalchar{+}}03]{GieHof+03Domains}
\bgroup\scshape{}G.~Gierz\egroup{}, \bgroup\scshape{}K.~H. Hofmann\egroup{},
  \bgroup\scshape{}K.~Keimel\egroup{}, \bgroup\scshape{}J.~D. Lawson\egroup{},
  \bgroup\scshape{}M.~Mislove\egroup{}, and \bgroup\scshape{}D.~S.
  Scott\egroup{}, \emph{Continuous lattices and domains}, \emph{Encyclopedia of
  Mathematics and its Applications} \textbf{93}, Cambridge University Press,
  Cambridge, 2003.

\bibitem[Rob12]{Rob12LimitsNCCW}
\bgroup\scshape{}L.~Robert\egroup{}, Classification of inductive limits of
  1-dimensional {NCCW} complexes,  \emph{Adv. Math.} \textbf{231} (2012),
  2802--2836.

\bibitem[Rob13]{Rob13Cone}
\bgroup\scshape{}L.~Robert\egroup{}, The cone of functionals on the {C}untz
  semigroup,  \emph{Math. Scand.} \textbf{113} (2013), 161--186.

\bibitem[RW10]{RorWin10ZRevisited}
\bgroup\scshape{}M.~R{\o}rdam\egroup{} and \bgroup\scshape{}W.~Winter\egroup{},
  The {J}iang-{S}u algebra revisited,  \emph{J. Reine Angew. Math.}
  \textbf{642} (2010), 129--155.

\bibitem[Thi13]{Thi13TopDimTypeI}
\bgroup\scshape{}H.~Thiel\egroup{}, The topological dimension of type {I}
  {$C^*$}-algebras,  in \emph{Operator algebra and dynamics}, \emph{Springer
  Proc. Math. Stat.} \textbf{58}, Springer, Heidelberg, 2013, pp.~305--328.

\bibitem[Thi20a]{Thi20arX:grSubhom}
\bgroup\scshape{}H.~Thiel\egroup{}, The generator rank of subhomogeneous
  \ca{s}, preprint (arXiv:2006.03624 [math.OA]), 2020.

\bibitem[Thi20b]{Thi20RksOps}
\bgroup\scshape{}H.~Thiel\egroup{}, Ranks of operators in simple \ca{s} with
  stable rank one,  \emph{Comm. Math. Phys.} \textbf{377} (2020), 37--76.

\bibitem[TV21]{ThiVil21arX:DimCu}
\bgroup\scshape{}H.~Thiel\egroup{} and \bgroup\scshape{}E.~Vilalta\egroup{},
  Covering dimension of {C}untz semigroups, preprint (arXiv:2101.04522
  [math.OA]), 2021.

\bibitem[Tom08]{Tom08ClassificationNuclear}
\bgroup\scshape{}A.~S. Toms\egroup{}, On the classification problem for nuclear
  \ca{s},  \emph{Ann. of Math. (2)} \textbf{167} (2008), 1029--1044.

\end{thebibliography}

\providecommand{\etalchar}[1]{$^{#1}$}

\end{document}